\documentclass[a4paper]{amsart} 

\usepackage{amsmath,amsthm,amssymb,amsfonts,mathrsfs,color,hyperref, mathtools,crop, graphicx, enumitem, todonotes}
\usepackage[left=3cm, right = 3cm]{geometry}

\theoremstyle{plain}
\begingroup
\newtheorem{theorem}{Theorem}[]
\newtheorem*{theorem*}{Theorem}
\newtheorem*{"theorem"}{``Theorem''}
\newtheorem{corollary}[theorem]{Corollary}

\newtheorem{lemma}[theorem]{Lemma}
\endgroup

\theoremstyle{definition}
\begingroup

\endgroup

\theoremstyle{remark}
\begingroup
\newtheorem{remark}[theorem]{Remark}
\newtheorem{example}[theorem]{Example}
\endgroup 

\numberwithin{equation}{section}
\newcommand{\N}{\mathbb N} 
 
\newcommand{\Z}{\mathbb Z} 
\newcommand{\R}{\mathbb R} 

\newcommand{\dist}{{\rm dist}}
\newcommand{\diam}{{\rm diam}}

\newcommand{\Per}{\mathrm{Per}}

\renewcommand{\H}{{\mathcal H}}

\renewcommand{\L}{{\mathcal L}}

\newcommand{\F}{{\mathcal F}}

\newcommand{\Ra} {\Rightarrow}

\renewcommand{\d}{\,\mathrm{d}}

\newcommand{\dx}{\,\mathrm{d}x}

\newcommand{\dz}{\,\mathrm{d}z}

\newcommand{\dt}{\,\mathrm{d}t}

\newcommand{\eps}{\varepsilon}
\newcommand{\average}{{\mathchoice {\kern1ex\vcenter{\hrule height.4pt
width 6pt depth0pt} \kern-9.7pt} {\kern1ex\vcenter{\hrule
height.4pt width 4.3pt depth0pt} \kern-7pt} {} {} }}
\newcommand{\avint}{\average\int}

 \allowdisplaybreaks
 
 \DeclareMathOperator*{\argmin}{argmin}

\begin{document}

\title[]{A note on spatially inhomogeneous Cahn-Hilliard energies}

\author{Stephan Wojtowytsch}
\address{Stephan Wojtowytsch\\
Department of Mathematics\\
University of Pittsburgh\\
Pittsburgh, PA 15213
}
\email{s.woj@pitt.edu}

\date{\today}


\begin{abstract}
In 2023, Cristoferi, Fonseca and Ganedi proved that Cahn-Hilliard type energies with spatially inhomogeneous potentials converge to the usual (isotropic and homogeneous) perimeter functional if the length-scale $\delta$ of spatial inhomogeneity in the double-well potential is small compared to the length-scale $\varepsilon$ of phase transitions. We give a simple new proof under a slightly stronger assumption on the regularity of $W$ with respect to the phase parameter. The simplicity of the proof allows us to easily find multiple generalizations in other directions, including stochastic potentials and potentials which may become zero outside the wells (`voids'). The theoretical results are complemented by numerical experiments.

Our main message is that, across a wide variety of settings, nothing that looks sufficiently homogeneous at the transition length scale between phases, affects the limiting behavior in the slightest. The notable exception is the setting of potentials with spatially varying potential wells, where a stronger scale separation $\delta \ll \eps^{3/2}$ is needed. Based on the analysis, we further provide a modified Modica-Mortola type energy which does not require such a scale separation. 

The theoretical analysis is complemented by numerical experiments.
\end{abstract}

\maketitle


\section{Introduction}

The celebrated result of Modica and Mortola \cite{modica1977esempio, modica1987gradient} is one of the most prominent examples in the theory of $\Gamma$-convergence with applications ranging from its classical setting in solid/solid phase transitions and two-phase fluids \cite{miranville2019cahn} to image segmentation \cite{ambrosio1990approximation, dondl2018approximation} and inspiring techniques in data science \cite{luo2017convergence, bosch2018generalizing, budd2021classification}. Heuristically, it provides a notion of perimeter for diffusely defined domains in which the sharp boundary is replaced by an interfacial region. 
More precisely, let $\Omega\subseteq\R^d$ be a bounded open set for $d\geq 2$ and
\begin{equation}\label{eq modica-mortola classical}
W(u) = \frac{(u^2-1)^2}4, \qquad \F_\eps:L^1(\Omega)\to [0,\infty), \qquad  \F_\eps(u) = \int_{\Omega} \frac\eps2 |\nabla u|^2 + \frac{W(u)}\eps\dx
\end{equation}
if $u\in H^1(\Omega) \cap L^4(\Omega)$ and $\F_\eps(u) = +\infty$ otherwise. Depending on context, the functional $\F_\eps$ is referred to as the Cahn-Hilliard energy, the Ginzburg-Landau energy, or the Modica-Mortola functional.

\begin{theorem}[Modica '77, Modica-Mortola '87]\label{theorem modica mortola}
Set $c_0 = \int_{-1}^1 \sqrt{2 \,W(u)}\d u = \frac{2\sqrt 2}3$. Then:
\begin{enumerate}
\item Let $u_\eps$ be a sequence of functions such that $\liminf_{\eps\to 0}\F_\eps(u_\eps)<\infty$. Then there exists $u\in BV(\{-1,1\})$ and a subsequence of $u_\eps$ (not relabelled) such that $u_\eps\to u$ in $L^1(\Omega)$.
\item $u_\eps$ converges to the $c_0$-fold of the perimeter functional in the sense of $\Gamma(L^1)$-convergence:
\begin{itemize}
\item If $u_\eps \to u$ in $L^1(\Omega)$, then $\liminf_{\eps\to 0} \F_\eps(u_\eps) \geq c_0\cdot \Per(\{u=1\})$ and
\item for every $u\in BV(\Omega;\{-1,1\})$ exists a sequence $u_\eps\in H_0^1(\Omega)\cap L^4(\Omega)$ such that $u_\eps \to u$ in $L^1(\Omega)$ and $\limsup_{\eps\to 0}\F_\eps(u_\eps)\leq c_0\cdot \Per_\Omega(\{u=1\})$ where $\Per_\Omega$ denotes the perimeter relative to $\Omega$.
\end{itemize}
\end{enumerate}
\end{theorem}

Intuitively, the double-well contribution $W(u)/\eps$ to the integrand in $\F_\eps$ forces $u$ to take values mostly in the `potential wells' at $\pm 1$, while the gradient contribution penalizes fast transitions. The terms balance precisely when $u$ transitions between the wells on a length scale $\eps$, and the energy is proportional to the area of the transition region. The constant $c_0$ is found by considering an optimal transition shape on the correct length scale. While we stated it for concrete $W$ for simplicity, the result holds for more general double-well potentials $W$ \cite{modica1987gradient}.

For an introduction to the Modica-Mortola functional and the theory of $\Gamma$-convergence in general, see e.g.\ \cite{braides2002gamma, dal2012introduction}. For an overview of $BV$-functions in general and sets of finite perimeter in particular, see e.g.\ \cite{giusti1984functions, ambrosio2000functions, evans2018measure}. 

The Ginzburg-Landau functional has seen extensive use both as a mathematical model and a computational tool in varied settings. Many generalizations of the functional haven been considered for vector-valued phase-parameters $u_\eps$ with multiple potential wells \cite{boyer2010cahn, bosch2015preconditioning, laux2018convergence, fischer2022quantitative}, higher order geometric energies \cite{du2005phase, roger2006modified, bretin2015phase, dondl2017phase, bellettini2023conjecture}, non-local versions of the squared gradient \cite{alberti1998phase, garroni2006singular, garroni2005gamma, garroni2006variational, cozzi2017nonlocal, dondl2019effect} and the convergence of the $L^2$- and $H^{-1}$-gradient flows of $\F_\eps$ to their respective sharp interface limits \cite{ilmanen1993convergence, mugnai2011convergence, laux2018convergence, fischer2020convergence, fischer2022quantitative, pego1989front, stoth1996convergence, garcke2005mechanical, novick2008cahn, hensel2024weak}.
In the context of the continuum mechanics of non-homogeneous media, there is an interest in extensions of the Modica-Mortola energy which are not isotropic and homogeneous \cite{ansini2003gradient, hagerty2018note, cristoferi2023homogenization_delta_small, cristoferi2023homogenization_delta_large}. A simple such model is
\begin{equation}\label{eq Feps}
\F_{\eps,\delta} (u) = \int_{\Omega} \frac\eps2 |\nabla u|^2 + \frac{W\big(x/\delta,\: u\big)}\eps\dx
\end{equation}
where $W$ is a double-well potential in $u$ and periodic in the spatial variable, i.e.

\begin{enumerate}
\item $W(z+ n, u) = W(z,u)$ for all $n\in \N^d$.
\item $W(z,u)\geq 0$ and $W(z,u) = 0$ if and only if $u=\pm 1$.
\end{enumerate}

Of the many anisotropic and/or inhomogeneous generalizations of \eqref{eq modica-mortola classical}, this one is attractive from a computational point of view since it retains the isotropic and quadratic structure of the highest-order term, simplifying numerical calculations. It is in particular more computationally tractable than a generalization where the spatial inhomogeneity affects the gradient such as explored in \cite{ansini2003gradient}.

Intuitively, there are three regimes to this model: 

\begin{enumerate}
\item $\delta \ll \eps$. In this case, homogenization of the spatial oscillations in the heterogeneous medium happen much faster than the collapse of the diffuse interface. Functions of low energy $u$ are expected to appear constant at the length-scale $\delta$, suggesting that the model is equivalent to a Modica-Mortola model with a homogenized potential 
\begin{equation}\label{eq whom}
W_{hom}(u) = \int_{[0,1]^d} W(z,u)\dz.
\end{equation}

\item $\eps\ll \delta$. In this case, the diffuse interface collapses to a hyper-surface much faster than the inhomogeneities on the spatial scale. It is expected that $\F_{\eps,\delta}$ resembles an isotropic sharp-interface perimeter in an inhomogeneous medium at the small length-scale $\delta\ll 1$, and a sharp-interface perimeter in a homogeneous but anisotropic medium at macroscopic length scales -- compare, e.g.\ \cite{ansini2010approximation}.

\item $\delta\sim\eps$. This is the most complicated case, as both effects occur at the same length scale and have to be considered simultaneously. Also here, the limiting functional is an anisotropic perimeter functional on the macroscopic scale.

\end{enumerate}

The final regime was resolved by in \cite{cristoferi2019homogenization, cristoferi2020erratum}, along with partial results in the first regime under the stronger assumption that $\delta\ll \eps^{3/2}$ \cite{hagerty2018note}. In full generality, the convergence result was obtained by Cristoferi, Fonseca and Ganedi in \cite{cristoferi2023homogenization_delta_small} for the first regime and for a different model in the second regime \cite{cristoferi2023homogenization_delta_large}. In the special case of a separable potential $W(x,u) = w(x)\,\widetilde W(u)$, the second regime has been studied by \cite{irenechoksiraghav}.

\begin{theorem}[Hagerty '18, Cristoferi, Fonseca and Ganedi '22: $\delta \ll \eps$]\label{theorem cfg}
Assume that $W$ is a double-well potential in $u$ in the sense that $W(x,u)\geq 0$ and $W(x, u) = 0$ if and only if $u=\pm 1$. Assume further that $W$ is measurable and one-periodic in $x$ (for fixed $u$), continuous in $u$ (for fixed $x$), and satisfies $W(x,u) \geq \gamma|u|$ if $|u|\geq M$ for some $M, \gamma>0$. Finally, assume that $\sup_{x\in \R^d, |u|\leq M} |W(x,u)| <\infty$. 

Let $\delta_\eps \ll \eps$ and $\F_\eps:= \F_{\eps,\delta_\eps}$ as in \eqref{eq Feps}.
Define $W_{hom}$ by \eqref{eq whom} and 
\begin{equation}\label{eq fhom}
\F_{\eps, hom} :L^1(\Omega) \to [0,\infty), \qquad  \F_{\eps,hom}(u) = \begin{cases} \int_\Omega \frac\eps2\,|\nabla u|^2 + \frac{W_{hom}(u)}\eps \dx & u \in H^1(\Omega)\\ +\infty &\text{else}\end{cases}.
\end{equation}
Then the following are true.
\begin{enumerate}
\item Let $u_\eps$ be such that $\liminf_{\eps\to 0} \F_\eps(u_\eps) < \infty$. Then there exists a subsequence of $u_\eps$ (not relabeled) and $u\in BV(\R^d, \{-1,1\})$ such that $u_\eps \to u$ in $L^1(\Omega)$.

\item $\Gamma(L^1)-\lim_{\eps\to 0} \F_\eps = \Gamma(L^1)-\lim_{\eps\to 0} \F_{\eps,hom} = c_{hom} \cdot \Per_\Omega$ for  $c_{hom} = \int_{-1}^1 \sqrt{2\,W_{hom}(u)}\d u$ where $\Per_\Omega$ denotes the perimeter relative to $\Omega$.
\end{enumerate}
\end{theorem}

In this article, we give a simple alternative proof of Theorem \ref{theorem cfg} under a different set of technical assumptions on $W$. Intuitively, we require higher regularity of $W$ in the phase parameter $u$, but relax other assumptions. We believe the proof technique to be of independent interest. As usual, the key problem is the $\Gamma-\liminf$ inequality. Essentially, rather than bounding $\liminf_{\eps}\F_\eps(u_\eps)$ from below for a general sequence $u_\eps$, we show that we only need to study the inequality for functions which satisfy good quantitative regularity bounds on short length scales. This is achieved by taking a short minimizing movements step from $u_\eps$ to $\hat u_\eps$, which reduces the value of $\F_\eps$ and increases the regularity while barely changing $u_\eps$ in $L^1(\Omega)$. With higher regularity, we find that it suffices to focus on the interfacial region, where we can establish the $\Gamma-\liminf$ inequality.

The simplicity of the proof allows for easy generalizations beyond the setting of Theorem \ref{theorem cfg}, for instance to the setting of stochastic rather than periodic homogenization and homogenization with multiple length-scales $\delta_1,\dots,\delta_n$. Partial results hold in the setting of inhomogeneous media whose potential wells may not be compatible in different phases, i.e.\ for potentials $W(x,u)$ whose minimizers in the $u$-variable may depend on $x$. 

The technical assumptions on $W$ are gathered in Section \ref{section assumptions}, along with examples of admissible potentials $W$ and a comparison to \cite{cristoferi2023homogenization_delta_small} in Section \ref{section examples}. Our version of Theorem \ref{theorem cfg} is stated precisely and proved in Section \ref{section main}, along with several variations and extensions. We provide some numerical illustrations in Section \ref{section numerical}.

\subsection{Notations}

By $B_r(x)$, we denote the Euclidean ball of radius $r$ around a center $x$. If $x=0$, we omit it: $B_r := B_r(0)$. The Lebesgue measure of a set $B$ is denoted by $|B|$. The Euclidean norm is denoted by $|\cdot|$ and the distance to a set $C$ is denoted as $\dist(x, C) = \inf_{y\in C}|x-y|$. For a set $E$, we consider the signed distance function $sd_E(x) = \dist(x, E^c) - \dist(x, E)$. If $u$ is in a space $X$ of functions (e.g.\ $W^{1,p}$) from $\Omega$ to $\R$ and $u$ takes values in a set $S\subset \R$, then we write $u\in X(\Omega; S)$. We adopt the convention that $C$ is a generic constant which does not depend on any parameters under consideration and whose value may change from line to line.

\section{Setting}

\subsection{Assumptions}\label{section assumptions}

In this section, we list our assumptions on the domain $\Omega$ and the potentials $W_\delta$ and $W_{hom}$. 
Let $\Omega\subseteq \R^d$ be open and bounded. 
We assume that for every $\delta>0$, there exist $N_{\delta} \in \mathbb N$ and a collection of open sets $Q_{i}^{\delta} \subseteq \Omega$ such that

\begin{enumerate}
\item $Q_i^\delta \cap Q_j^\delta = \emptyset$ for all $i\neq j$.
\item The sets $Q_i^\delta$ are regular on length scale $\delta$:
\begin{enumerate}
\item There exists $R>0$ such that $\diam(Q_i^\delta)\leq R\delta$ for all $\delta, i$ and
\item There exists $c_P>0$ such that the Poincar\'e inequality
\[
\int_{Q_i^\delta} \big|u - \langle u\rangle_{Q_i^\delta}\big|^2\dx \leq c_P\delta^2\int_{Q_i^\delta}\|\nabla u\|^2\dx
\]
holds for all $i, \delta$ where $\langle u\rangle_{Q_i^\delta}:= \frac1{|Q_i^\delta|}\int_{Q_i^\delta}u\dx$.

\end{enumerate}
\item $\L^d\big(\Omega\setminus \bigcup_{i=1}^{N_\delta} Q_i^\delta\big) =0$.
\end{enumerate}

If all $Q_i^\delta$ are e.g.\ hypercubes of side-length $\sim \delta$ (or e.g.\ hexagons or another regular shape which tesselates the plane), the uniform Poincar\'e inequality holds by scaling. In principle, the condition could be dropped for $Q_i^\delta$ within distance $\sim\delta$ to the boundary if we assume the lower bound on $W$ rather than the scale separation. Otherwise, the condition can be read as shape (and thus also boundary) regularity for the sets $Q_i^\delta$.

For all $\delta>0$, we consider a spatially varying double-well potential $W_\delta:\Omega\times\R\to\R$ and a `homogenized' double-well potential $W_{hom}:\R\to\R$. We make the following assumptions on $W_\delta$.

\begin{enumerate}\setcounter{enumi}{3}

\item $W_\delta$ is jointly measurable as a function of $x$ and $u$.

\item\label{condition scale separation}
 One of the following holds:
\begin{enumerate}
\item There exist $C >0$ and $\gamma \in (0,1)$ such that $W_\delta(x,u) \geq -C\delta$ if $|u|\geq \gamma$ or
\item We assume that the scale separation $\delta \ll \eps^{3/2}$ holds below. 
\end{enumerate}

\item There exists $M\geq 1$ such that $W_\delta(x,u) \geq W_\delta (x, \:M\,u/|u|)$ if $|u|\geq M$.

\item \label{condition lipschitz}
There exist $p\geq 2$ (for $d\leq 3$) and $p>d/2$ (for $d\geq 4$) and $L_\delta \in L^p(\Omega; (0,\infty))$ such that $W_\delta$ satisfies the Lipschitz condition
\[
|W_\delta(x,u_1) - W_\delta (x, u_2)|\leq L_\delta(x)\,|u_1-u_2|\qquad \forall\ x\in\Omega, \quad u_1, u_2\in [-M, M].
\] 
Furthermore, the bound
\[
\frac1{|Q_i^\delta|} \int_{Q_i^\delta} L_\delta(x)^p \dx \leq C
\]
holds for all $\delta>0$ and $i= 1, \dots, N_\delta$.
\end{enumerate}

For $W_{hom}$, we assume the following.

\begin{enumerate}\setcounter{enumi}{7}
\item $W_{hom}$ is Lipschitz-continuous in $[-1,1]$.
\item $W_{hom}(-1) = W_{hom}(1) = \min_{u\in\R} W_{hom}(u) = 0$.
\item $W_{hom}(u)>0$ for $z\in (-1,1)$ and there exist $\gamma\in (0,1)$ such that $W_{hom}$ is monotone decreasing on $(\gamma,1)$ and monotone increasing on $(-1,-\gamma)$.
\end{enumerate}

Finally, we require a compatibility condition for $W_\delta$, $W_{hom}$ and $Q_i^\delta$:

\begin{enumerate}\setcounter{enumi}{10}    
\item Let $\mathcal U_\delta$ denote the set of vectors $u\in \R^{N_\delta}$ such that $u_i\in [-M, M]$ for all $i=1,\dots, N_\delta$.  $W_\delta$ satisfies
\begin{equation}\label{eq homogenization bound}
\sup_{u\in \mathcal U_\delta}  \sum_{i=1}^{N_\delta} \left|\int_{Q^\delta_i} W_\delta(x, u_i) - W_{hom}(u_i)\dx \right| \leq C\delta.
\end{equation}
\end{enumerate}

Essentially, $\mathcal U_\delta$ is a representation of functions which are constant on every $Q_i^\delta$ in the partition of $\Omega$. The compatibility condition therefore deals with functions which are constant on small scales. 

The scale separation in Condition \eqref{condition scale separation} was first considered in \cite{hagerty2018note} as a technical tool. In the present setting, we manage to relax it for a large class of potentials, and we show its necessity in other cases. The precise scaling can be understood from Example \ref{example voids} and Lemma \ref{lemma eps32} below. Heuristically, it stems from a term capturing small scale oscillations which has magnitude $\delta^2/\eps^2 \cdot 1/\eps$. In most models, such a term would be small and only relevant in the vicinity of the transition region where $u$ is not close to pure phase $\pm 1$, which has measure $\sim \eps$. However, if $W_\delta$ can be negative close to pure phase, there may be a bulk energy contribution and the condition $\delta^2/\eps^3 \ll 1$ is required.

We define the functionals $\F_{\eps, \delta}$ and $\F_{\eps, hom}$ as
\begin{equation}\label{eq our modica mortola functionals}
\F_{\eps,\delta}(u) = \int_\Omega \frac\eps2 \|\nabla u\|^2 + \frac{W_\delta(x,u)}\eps\dx, \qquad \F_{\eps, hom}(u) = \int_\Omega \frac\eps2 \|\nabla u\|^2 + \frac{W_{hom}(u)}\eps\dx.
\end{equation}

\subsection{Discussion and Examples}\label{section examples}

The assumptions in \cite{cristoferi2023homogenization_delta_small} are somewhat different. In both settings, $W_\delta$ is a measurable function of $x$ and $u$, but here we assume $W_\delta$ to be Lipschitz-continuous in $u$ while \cite{cristoferi2023homogenization_delta_small} only requires continuity in $u$. On the other hand, \cite{cristoferi2023homogenization_delta_small} makes two strong assumptions which we relax significantly:

\begin{enumerate}
\item $W_\delta(x, u) = \widetilde W(x/\delta, u)$ where $\widetilde W$ is $1$-periodic in all coordinate directions in $x$.

\item $W_\delta\geq 0$ and $W_\delta(x,u) = 0$ if and only if $u= \pm 1$, independently of $x$.
\end{enumerate}

By comparison, we are able to deal with structures which are not necessarily periodic in space, but may have a sufficiently regular stochastic structure. We also allow potentials $W_\delta$ whose minimizers are attained at different points as in \cite{cristoferi2023homogenization_delta_small}, and we allow a setting with `voids' where $W_\delta(x, \cdot) \equiv 0$ for $x$ in an open set.

\begin{example}[Multiple materials and voids]\label{example voids}
We can consider potentials of the form $W_\delta(x,u) = W(x/\delta, u)$ where
\[
W(x,u) = a(x)\,\big(u^2 - b(x)\big)^2 - c(x)
\]
and $a,b,c$ are $1$-periodic functions (or more generally, periodic with respect to a lattice). The homogenized potential is
\begin{align*}
W_{hom}(u) &= \left(\int_{(0,1)^d} a(x)\dx\right) \,u^4 -2\left(\int_{(0,1)^d}a(x)b(x) \dx\right)\,u^2 + \int_{(0,1)^d} a(x)b^2(x) - c(x)\dx\\
	&= \left(\int_{(0,1)^d} a(x)\dx\right) \big(u^2-1\big)^2
\end{align*}
if the coefficient functions $a,b,c$ satisfy
\[
\int_{(0,1)^d}ab\dx = \int_{(0,1)^d} a b^2-c\dx =  \int_{(0,1)^d}a\dx.
\]
The case where $a, b, c$ are piecewise constant functions is of particular interest and can be thought of as a model for several hetereogeneous phases if $b$ is non-constant (see e.g.\ \cite{cristoferi2023homogenization_delta_large}), or as a model for a material with voids if $b\equiv 1$, $c\equiv 0$ and $a$ takes values in $\{0,1\}$.

If $b\equiv 1$, we can select $c\equiv 0$ and Condition \eqref{condition scale separation} is met as $W_\delta\geq 0$. Consider on the other hand the one-dimensional model where $a\equiv 1$,
\[
b(x) = \frac12\,\chi(x) + \frac32\,\big(1-\chi(x)\big), \qquad c = \frac12\bigg(\left(\frac32\right)^2 + \left(\frac12\right)^2\bigg)-1 = \frac{9}8 + \frac18 -1= \frac 14.
\] 
and $\chi$ is the 1-periodic extension of $1_{(0, 1/2)}$. In this case $W_\delta(x, \pm 1) = -1/4 <0$ and Condition \eqref{condition scale separation} is not met unless $\delta \ll \eps^{3/2}$. Let us illustrate that this is indeed necessary. Consider 
\[
\phi_{\eps, \delta}:(0,1)\to\R, \qquad \phi_{\eps,\delta}(x) = 1 - \left(\frac\delta\eps\right)^2\,\psi\left(\frac x\delta\right)
\]
where $\psi$ is a smooth and $1$-periodic function which is positive on $(1/2, 1)$ and negative on $(0,1/2)$. By Taylor expansion, we have
\begin{align*}
\int_0^1 &\frac\eps2\,\big|\phi_{\eps,\delta}(x)\big|^2 + \frac{W(x/\delta, \,\phi_{\eps,\delta}(x))}\eps\dx\\
	&= \int_0^1 \frac\eps2\left(\frac\delta\eps\right)^4\,\frac1{\delta^2}\big|\psi'\big|^2\left(\frac x\delta\right) +\frac1\eps\left( \frac14- \frac14 + 2\left(1 -b(x/\delta)\right) \left(\frac\delta\eps\right)^2\,\psi\left(\frac x\delta\right) + O\left((\delta/\eps)^4\right)\right) \dx \\
	&= \frac{\delta^2}{\eps^3} \int_0^1\left( \frac12\,|\psi'|^2-|\psi|\right)(x/\delta)\dx +O\left(\frac{\delta^4}{\eps^5}\right).
\end{align*}
Scaling $\psi$ by a small positive factor $\alpha$, we guarantee that $\int \frac12 (\alpha\psi')^2 - |\alpha\psi|\dx < 0$, meaning that the energy is not bounded from below uniformly if $\eps^{3/2}\ll \delta \ll \eps$. Thus in this case, the scale separation $\delta\ll \eps^{3/2}$ must be assumed in Condition \eqref{condition scale separation}. See also Lemma \ref{lemma eps32} and Corollary \ref{corollary modified modica mortola} for further insight.
\end{example}

All results apply if the average of $W$ over a `periodic cell' differs from $W_{hom}$ by $O(|Q_i^\delta|\cdot \delta)$, for instance if there is no $u\in\R$ such that the average of $W_1(x,u)$ over a period cell vanishes exactly, e.g.\ because $a_\delta(x), c_\delta(x)$ mildly depend on $\delta$ or because the potential wells vary as $b_\delta(x) = 1+ O(\delta)$. See also Remark \ref{remark multiple length scales} about potentials with multiple incompatible length scales $\delta_1,\dots, \delta_n$.

\begin{example}[Stochastic Materials]\label{example stochastic materials}
Assume for simplicity that $\Omega = (0,1)^d$ and $\delta = \delta_n = 1/n$. For the collection of small sets covering $\Omega$, we choose cubes of side-length $\delta$: $Q_i^\delta = \prod_{l=1}^d (i_l\delta, (i_l+1)\delta)$ where $i\in \{0,1,\dots, n-1\}^d$. We further partition every cube $Q_i^\delta$ into $m^d$ smaller cubes $Q_{ij}^\delta$ of side length $\delta/m$. For every one of these smaller cubes, we choose a random variable $q_{ij}^\delta$ which takes values in $[0,1]$ with expected value $p\in(0,1)$ and we set $W_\delta(x,z) = q_{ij}^\delta\,\widetilde W(u) $ for $x\in Q_{ij}^\delta$. We may choose any locally Lipschitz-continuous double-well potential $\widetilde W$, e.g. $\widetilde W(u) = |1-u^2|$ or similar. The `homogenized potential' is $W_{hom}(u) = p\,\widetilde W(u)$. We have 
\[
\frac1{|Q_i^\delta|} \int_{Q_i^\delta} W_\delta(x,u) \dx - W_{hom}(u) = \left(\frac1{m^d}\sum_jq_{ij}^\delta - p\right)\,|1-u^2|.
\]
We conclude that
\begin{align*}
\sup_{u\in \mathcal U_\delta}\left|\sum_i\int_{Q_{i}^\delta} W_\delta(u_i) - W_{hom}(u_i)\dx \right| &\leq \sum_i |Q_i^\delta|\,\left|\frac1{m^d}\sum_jq_{ij}^\delta - p\right|
	= \frac1{n^d} \sum_i \left|\frac1{m^d}\sum_jq_{ij}^\delta - p\right|.
\end{align*}
We will show that this term is small, both in expectation and with high probability. The techniques are standard, but since we do not expect all of our readers to be familiar with high-dimensional probability, we provide all details. By Hoeffding's Lemma \cite[Lemma 3.6]{van2014probability}, the random variables $q_{ij}^\delta$ are sub-Gaussian with variance proxy $(1-0)^2/4 = 1/4$, i.e.\
\[
\psi(\lambda):= \mathbb E\big[\exp\big(\lambda (q_{ij}^\delta -p))\big)\big] \leq \exp(\lambda^2/8). 
\]
If the values $q_{ij}^\delta$ are independent, the moment-generating function $\psi$ factorizes and we find that $\frac1{m^d}\sum_j q_{ij}^\delta- p$ is $1/(4m^d)$-sub-Gaussian, i.e.\
\[
\mathbb E\left[\exp\left(\frac\lambda{m^d}\sum_j\big(q_{ij}^\delta -p)\big)\right)\right] \leq \exp\left(\frac{\lambda^2}{8\,m^d}\right).
\]
In expectation over the realization of $q_{ij}^\delta$, by Jensen's inequality we have 
\begin{align*}
\mathbb E\bigg[\sup_{u\in \mathcal U_\delta}&\left|\sum_i\int_{Q_{i}^\delta} W_\delta(u_i) - W_{hom}(u_i)\dx \right|\bigg]
	\leq \mathbb E\left[\sum_i |Q_i^\delta|\,\left|\frac1{m^d}\sum_jq_{ij}^\delta - p\right| \right]\\
	&= \mathbb E \left[ \left|\frac1{m^d}\sum_jq_{0j}^\delta - p\right| \right]\\
	&\leq \frac1\lambda\,\log\left( \mathbb E\left[ \exp\left(\frac{\lambda}{m^d}\left|\sum_j(q_{0j}^\delta - p)\right|\right) \right]\right)\\
	&\leq \frac1\lambda\,\log\left( \mathbb E\left[\exp\left( \frac{\lambda}{m^d}\sum_j(q_{0j}^\delta - p)\right) + \exp\left(- \frac{\lambda}{m^d} \sum_j(q_{0j}^\delta - p)\right)  \right]\right)\\
	&\leq \frac1\lambda \log\left( 2\,\exp\left(\frac{\lambda^2}{8m^d}\right)\right)\\
	&= \frac{\lambda}{8m^d} + \frac{\log 2}\lambda
\end{align*}
for any $\lambda>0$. The bound becomes minimal when 
\[
\frac1{8m^d} - \frac{\log 2}{\lambda^2} =0 \qquad\Ra \quad \mathbb E\bigg[\sup_{u\in \mathcal U_\delta}\left|\sum_i\int_{Q_{i}^\delta} W_\delta(u_i) - W_{hom}(u_i)\dx \right|\bigg] \leq 2\sqrt{\frac{\log 2}{8m^d}} \leq m^{-d/2}.
\]
In the same way, we can compute a high probability Chernoff-style bound
\begin{align*}
\mathbb P\bigg(\frac1{n^d} &\sum_i \left|\frac1{m^d}\sum_jq_{ij}^\delta - p\right| > t\bigg) \leq e^{-\lambda t} \mathbb E\left[\exp\left(\frac\lambda {n^d} \sum_i \left|\frac1{m^d}\sum_jq_{ij}^\delta - p\right|\right)\right]\\
	&= e^{-\lambda t} \mathbb E\left[\prod_i\exp\left(\left|\frac\lambda{(nm)^d}\sum_j(q_{ij}^\delta - p)\right|\right)\right]\\
	&= e^{-\lambda t}\prod_i \mathbb E\left[\exp\left(\left|\frac\lambda{(nm)^d}\sum_j(q_{ij}^\delta - p)\right|\right)\right]\\
	&\leq e^{-\lambda t}\prod_i \mathbb E\left[\exp\left(\frac\lambda{(nm)^d}\sum_j(q_{ij}^\delta - p)\right) + \exp\left(-\frac\lambda{(nm)^d}\sum_j(q_{ij}^\delta - p)\right)\right]\\
	&\leq e^{-\lambda t}\prod_i 2\,\exp\left(\frac{\left(\lambda/ {n^d}\right)^2}{8m^d}\right)\\
	&= e^{-\lambda t} \exp\left(\sum_i \left(\log 2 + \frac{\lambda^2}{8n^{2d}m^d}\right)\right)\\
	&= \exp\left(\frac{\lambda^2}{8n^dm^d}- \lambda t + n^d\log 2\right).
\end{align*}
Given $t>0$, the minimum of the polynomial is attained when $\lambda = 4n^dm^d t$ with the value
\[
\frac{\lambda^2}{8n^dm^d}- \lambda t + n^d\log 2 = - 2n^dm^d\,t^2 + n^d\log 2 = n^d\left(\log 2-m^dt^2\right).
\]
For $t = m^{-d/2}$, we see that 
\begin{align*}
\mathbb P\left(\sup_{u\in \mathcal U_\delta}\left|\sum_i\int_{Q_{i}^\delta} W_\delta(u_i) - W_{hom}(u_i)\dx \right| \geq m^{-d/2}\right) \leq \exp\left(\big(\log 2 - 1\big) n^d\right) \leq \exp(-0.3\cdot n^d).
\end{align*}
Thus, if $n, m$ are moderately large, then the discrepancy between $W_\delta$ and $W_{hom}$ is sufficiently small with very high probability. The bounds become stronger in higher dimension since a partition of a cube on a given length-scale contains more small cubes.
\end{example}

The example applies, for instance, to materials where $q_{ij}^\delta$ is distributed uniformly in $[0,1]$ (with $p=1/2$) or to choosing $q_{ij}^\delta = 1$ with probability $p$ and $q_{ij}^\delta =0$ with probability $1-p$. Note that for potentials of the separable structure $W(x,u) = w(x)\,\widetilde W(u)$ with a deterministic or random weight function $w$ as in Examples \ref{example voids} and \ref{example stochastic materials}, it is entirely possible that the the set $w=0$ is connected and that $w$ is zero on the vast majority of the domain, as long as it averages out on a sufficiently short length scale.

\begin{example}[Refinement at the boundary]\label{example multiple lengthscales}
The length scale $\delta$ is the `longest' microscale on which $W_\delta$ oscillates, but not necessarily the only one. 
Our framework also applies to potentials $W_\delta$ whose microlengthscale becomes smaller as we approach the domain $\Omega$. For instance, assume that $\widetilde W:(\R^d/\Z^d)\times \R\to\R$ is a potential which satisfies all necessary continuity and measurability assumptions. Then for $x\in(0,1)^d$ we can define 
\[
\delta(x) = \max \big\{2^k\delta : k\in \mathbb Z\text{ and }2^k\delta \leq \min\{x_1,\dots, x_d, 1-x_1, \dots, 1-x_d\} 
\big\}
\]
and the potential
\[
W_{\delta}: [0, 1]^d\to\R, \qquad W_{\delta}(x,u) =\widetilde W\left(\frac{x}{\min\{\delta, \delta(x)\}}, \:u\right).
\]
The refinement at the boundary could also be terminated after a finite number of scales.
\end{example}

With the same techniques, we can also allow for potentials with multiple small length scales: 
\[
\delta = (\delta_1, \dots, \delta_n), \qquad W_\delta(x,u) = \sum_{i=1}^n  W_{i}\left(\frac x{\delta_i}, u\right), \qquad \max_i \delta_i \ll \eps
\]
If for instance $\delta_1/\delta_2\notin \mathbb Q$, then there is no shared periodic cell for the material, but the same proofs go through -- see Remark \ref{remark multiple length scales}.

\begin{figure}
\begin{center}
\begin{tikzpicture}[scale = 5]
\draw[very thick](.001, .001) --++ (0, .998) -- ++ (.998, 0) -- ++ (0, -.998) -- ++ (-.998, 0) ;

\foreach \x in {1,...,4}{
	\draw[](\x/5, 0) -- ++ (0,1);
	\draw[](0,\x/5) -- ++ (1,0);	
}

\foreach \x in {1, ..., 4}{
	\draw[]({2^(-\x)/5}, 0) -- ++ (0,1);
	\draw[]({1-2^(-\x)/5}, 0) -- ++ (0,1);
	\draw[](0, {2^(-\x)/5}) -- ++ (1,0);
	\draw[](0, {1-2^(-\x)/5}) -- ++ (1,0);

	\foreach\y in {0,..., 4}{
		\draw[]({\y/5 +2^(-\x)/5}, 0) -- ++ (0, {2^(1-\x)/5});
		\draw[]({(\y+1)/5 -2^(-\x)/5}, 0) -- ++ (0, {2^(1-\x)/5});
		\draw[]({\y/5 +1/10+2^(-\x)/5}, 0) -- ++ (0, {2^(1-\x)/5});
		\draw[]({(\y+1)/5 -1/10-2^(-\x)/5}, 0) -- ++ (0, {2^(1-\x)/5});
	}
	\foreach\y in {0, ..., 9}{
		\draw[](\y/10+1/20-1/80,0)--++(0,1/40);
		\draw[](\y/10+1/20+1/80,0)--++(0,1/40);
	}
	\begin{scope}[yshift = 1cm, yscale = -1]
	\foreach\y in {0,..., 4}{
		\draw[]({\y/5 +2^(-\x)/5}, 0) -- ++ (0, {2^(1-\x)/5});
		\draw[]({(\y+1)/5 -2^(-\x)/5}, 0) -- ++ (0, {2^(1-\x)/5});
		\draw[]({\y/5 +1/10+2^(-\x)/5}, 0) -- ++ (0, {2^(1-\x)/5});
		\draw[]({(\y+1)/5 -1/10-2^(-\x)/5}, 0) -- ++ (0, {2^(1-\x)/5});
	}
	\foreach\y in {0, ..., 9}{
		\draw[](\y/10+1/20-1/80,0)--++(0,1/40);
		\draw[](\y/10+1/20+1/80,0)--++(0,1/40);
	}
	\end{scope}

	\begin{scope}[xscale=-1, rotate=90]
	\foreach\y in {0,..., 4}{
		\draw[]({\y/5 +2^(-\x)/5}, 0) -- ++ (0, {2^(1-\x)/5});
		\draw[]({(\y+1)/5 -2^(-\x)/5}, 0) -- ++ (0, {2^(1-\x)/5});
		\draw[]({\y/5 +1/10+2^(-\x)/5}, 0) -- ++ (0, {2^(1-\x)/5});
		\draw[]({(\y+1)/5 -1/10-2^(-\x)/5}, 0) -- ++ (0, {2^(1-\x)/5});
	}
	\foreach\y in {0, ..., 9}{
		\draw[](\y/10+1/20-1/80,0)--++(0,1/40);
		\draw[](\y/10+1/20+1/80,0)--++(0,1/40);
	}
	\end{scope}
	
		\begin{scope}[xshift = 1cm, rotate=90, ]
	\foreach\y in {0,..., 4}{
		\draw[]({\y/5 +2^(-\x)/5}, 0) -- ++ (0, {2^(1-\x)/5});
		\draw[]({(\y+1)/5 -2^(-\x)/5}, 0) -- ++ (0, {2^(1-\x)/5});
		\draw[]({\y/5 +1/10+2^(-\x)/5}, 0) -- ++ (0, {2^(1-\x)/5});
		\draw[]({(\y+1)/5 -1/10-2^(-\x)/5}, 0) -- ++ (0, {2^(1-\x)/5});
	}
	\foreach\y in {0, ..., 9}{
		\draw[](\y/10+1/20-1/80,0)--++(0,1/40);
		\draw[](\y/10+1/20+1/80,0)--++(0,1/40);
	}
	\end{scope}

\draw[<->](-.05, .2) -- ++ (0, .2);
\node[left] at(-.05,  .3){$\delta$};
}
\end{tikzpicture}
\hspace{5mm}
%
%
%
\begin{tikzpicture}[scale = 5]
\draw[very thick](0,0) -- ++(1,0) -- ++ (0,1) -- ++ (-1, 0) -- ++ (0, -1);
\foreach\x in {0, ..., 40}{
	\draw[](0,\x/40) -- ++ (1,0);
	\draw[](\x/40,0) -- ++ (0,1);
}
\foreach\x in {0,..., 39}{
	\foreach \y in {0, ..., 39}
	\fill[opacity = (rand+1)/2 ](\x/40, \y/40) -- ++(1/40, 0) --++(0, 1/40) -- ++ (-1/40,0);
}
\draw[<->] (1.05, .2) --++(0, .2);
\node[right] at (1.05, .3) {$\delta$};
\end{tikzpicture}
\end{center}

\caption{\label{figure not periodic}
An illustration of non-periodic geometries to which our results apply. Left: The geometry of Example \ref{example multiple lengthscales} with boundary refinement. Every square is a rescaled copy of the periodic cell of the potential. Right: The weight grid for a stochastic material as in Example \ref{example stochastic materials} with $q_{ij}^\delta$ distributed uniformly in $[0,1]$.
}
\end{figure}
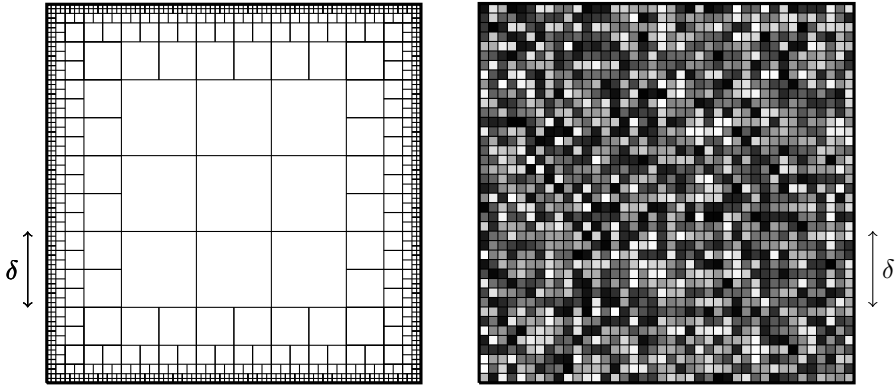

\begin{example}[Rare inclusions]
In previous examples, we considered the situation in which there may be a small discrepancy between $W_{hom}$ and the average of $W_\delta$ over many periodic cells. In the opposite regime, we can consider a potential that coincides with $W_{hom}$ except on randomly or deterministically placed `inclusion sites' which we can picture as a harder or softer material. In dimension $d$, the same techniques apply for instance if the modification to $W$ is of magnitude $1$ on $N_{d,\delta} \sim \delta^{1-d}$ balls or cubes of diameter $\sim \delta$. The details are left to the reader.
\end{example}

\begin{example}[Less regular potentials in $u$]
Consider the potential $W_\delta(x, u) = W_{hom}(x,u) = |1-u^2|^\alpha= |1-u|^\alpha\,|1+u|^\alpha$. Inside of $[-1,1]$, the potential is H\"older continuous with exponent $\alpha$, but not with any exponent $\beta>\alpha$. As we require Lipschitz continuity, we are constrained to the range $\alpha\geq 1$, while the results of \cite{cristoferi2023homogenization_delta_small} apply to this potential for the full range $\alpha>0$ and spatially inhomogenous generalizations of the form
\[
W_1(x) = c(x)\,\big|1-u^2\big|^{\alpha(x)}
\]
for strictly positive and measurable periodic functions $c(x), \alpha(x)$.
\end{example}

\section{Main results}\label{section main}

We now come to the statement and proof of our main results for the homogenization of materials with phase boundaries in the regime $\delta\ll \eps$. First, we consider two auxiliary statements which easily implies Theorem \ref{theorem cfg} under the alternative assumptions of Section \ref{section assumptions}. 

Denote the truncation of a function at $\pm M$ by $T_Mu = Mu/ \max\{|u|, M\}$ and we write that $u_\eps \xrightarrow {d_M} u$ if $T_Mu_\eps\to T_M u$ in $L^1(\Omega)$.

\subsection{A preliminary analysis}
The following Lemma is a more careful statement of results obtained in Lemma 3.5 of \cite{hagerty2018note}.

\begin{lemma}\label{lemma eps32}
For all $\alpha>0$, the following bound holds: 
\begin{align*}
\int_{Q_i^\delta} \frac\eps 2 &\,\|\nabla u\|^2 + \frac{W_\delta(x,u)}\eps \dx 
	\geq \int_{Q_i^\delta} \frac{(1-C_P\alpha)\eps} 2 \,\|\nabla u\|^2 + \frac{W_{hom}(u) }\eps\dx\\
	&\hspace{4cm}- C\,\left(\frac{\delta}{\eps}\right)^2\frac 1{\alpha\eps}\,|Q_i^\delta| - \left|\int_{Q_i^\delta} \frac{W_\delta(x,\langle u\rangle) - W_{hom}(\langle u\rangle)}\eps\dx\right|.
\end{align*}
\end{lemma}

\begin{proof}
Let $\langle u\rangle = \frac1{|Q_i^\delta|}\int_{Q_i^\delta} u\dx$. We compute
{\small 
\begin{align*}
\int_{Q_i^\delta} \frac\eps 2 &\,\|\nabla u\|^2 + \frac{W_\delta(x,u)}\eps \dx\\
	 &= \int_{Q_i^\delta} \frac\eps 2 \,\|\nabla u\|^2 + \frac{W_\delta(x,u) - W_\delta(x,\langle u\rangle) + W_\delta(x,\langle u\rangle)}\eps \dx\\
	 &\geq \int_{Q_i^\delta} \frac\eps 2 \,\|\nabla u\|^2 + \frac{W_\delta(x,u) - W_\delta(x,\langle u\rangle) + W_{hom}(\langle u\rangle)}\eps \dx - \left|\int_{Q_i^\delta} \frac{W_\delta(x,\langle u\rangle) - W_{hom}(\langle u\rangle)}\eps\dx\right|\\
	 &\geq \int_{Q_i^\delta} \frac\eps 2 \,\|\nabla u\|^2 + \frac{W_{hom}(\langle u\rangle) - L_\delta(x)\,|u - \langle u\rangle|}\eps \dx - \left|\int_{Q_i^\delta} \frac{W_\delta(x,\langle u\rangle) - W_{hom}(\langle u\rangle)}\eps\dx\right|\\
	 &= \int_{Q_i^\delta} \frac\eps 2 \,\|\nabla u\|^2 + \frac{W_{hom}(u) }\eps\dx - \int_{Q_i^\delta} \frac{L_\delta(x) + L}{\eps}|u-\langle u\rangle| \dx - \left|\int_{Q_i^\delta} \frac{W_\delta(x,\langle u\rangle) - W_{hom}(\langle u\rangle)}\eps\dx\right|.
\end{align*}
}
We analyze the middle term:
{\small
\begin{align*}
\int_{Q_i^\delta} \frac{L_\delta(x) + L}{\eps}|u-\langle u\rangle| \dx 
	\leq \int_{Q_i^\delta} \frac{\big(L_\delta(x) + L\big)^2\delta^2} {2\alpha\eps^3}+ \frac{\alpha\eps}{2\delta^2} \,|u-\langle u\rangle|^2 \dx
	\leq C\,\frac{\delta^2}{\eps^3\alpha}\,|Q_i^\delta| + \int_{Q_i^\delta} \frac{C_P\alpha\eps}{2} \,\|\nabla u\|^2 \dx
\end{align*}
}
for any $\alpha>0$. In total:
{\small
\begin{align*}
\int_{Q_i^\delta}& \frac\eps 2 \,\|\nabla u\|^2 + \frac{W_\delta(x,u)}\eps \dx \\&
	\geq \int_{Q_i^\delta} \frac{(1-C_P\alpha)\eps} 2 \,\|\nabla u\|^2 + \frac{W_{hom}(u) }\eps\dx - C\,\left(\frac{\delta}{\eps}\right)^2\frac 1{\alpha\eps}\,|Q_i^\delta| - \left|\int_{Q_i^\delta} \frac{W_\delta(x,\langle u\rangle) - W_{hom}(\langle u\rangle)}\eps\dx\right|.
\end{align*}
}
\end{proof}

As observed in \cite{hagerty2018note}, this result immediately implies convergence if $\delta \ll \eps^{3/2}$.

\begin{corollary}
Assume that $W$ satisfies the conditions in Section \ref{section assumptions} and $\delta_\eps \ll \eps^{3/2}$. Then
\[
\Gamma(L^1)-\lim_{\eps\to 0^+} \F_{\eps,\delta_\eps} = \Gamma(L^1)-\lim_{\eps\to 0^+} \F_{\eps,hom}.
\]
\end{corollary}

\begin{proof}
Observe that $\F_{\eps,\delta}(T_Mv) \leq \F_{\eps,\delta}(v)$ for any $v$ by the assumptions on $W_\delta$ and \cite[Satz 5.20]{dobrowolski2010angewandte}.
Let $\alpha_\eps>0$ be such that $\alpha_\eps \to 0^+$ and $\frac{\delta_\eps^2}{\eps^3\,\alpha_\eps}\to 0^+$. Then
\begin{align*}
\liminf_{\eps\to 0^+}\F_{\eps,\delta_\eps}(u_\eps) &\geq \liminf_{\eps\to 0^+} \F_{\eps,\delta_\eps}(T_Mu_\eps)\\
	&= \liminf_{\eps\to 0^+} \int_{\Omega } \frac{\eps} 2 \,\|\nabla T_Mu_\eps\|^2 + \frac{W_{hom}(x,T_Mu_\eps) }\eps\dx\\
	&\geq \liminf_{\eps\to 0^+}\bigg[\sum_{i=1}^{N_{\delta_\eps}} \int_{Q_i^{\delta_\eps}} \frac{(1-C_P\alpha_\eps)\eps} 2 \,\|\nabla u\|^2 + \frac{W_{hom}(x,u) }\eps\dx\\
	&\qquad - C\,\frac {\delta_\eps^2}{\eps^3\,\alpha\eps}\,\sum_{i=1}^{N_{\delta_\eps}}|Q_i^{\delta_\eps}| - \sum_{i=1}^{N_{\delta_\eps}} \frac1\eps \bigg|\int_{Q_i^\delta} W_\delta(x,\langle u\rangle) - W_{hom}(\langle u\rangle)\dx\bigg|\bigg]\\
	&\geq \liminf_{\eps\to 0^+}\bigg[ (1-C_P\alpha_\eps)\,\F_{\eps,hom}(T_Mu_\eps) - \frac{\delta_\eps^3}{\eps^2\alpha_\eps} \,|\Omega| - C\frac{\delta_\eps}\eps\bigg]\\
	&\geq \liminf_{\eps \to 0^+}\F_{\eps,hom}(T_Mu_\eps).
\end{align*}
It is easy to see that $\F_{\eps,\delta_\eps}$ and $\F_{\eps,hom}$ have the same limit along a recovery sequence, which is $O(1/\eps)$-Lipschitz at the interface and (essentially) constant elsewhere. This holds even without the scale separation assumption. For details, see the proof of Theorem \ref{theorem main}.
\end{proof}

In the general case, we need an additional argument which allows us to focus only on the interface between the phases $\{u\approx -1\}$ and $\{u\approx 1\}$ when establishing the liminf-inequality. The functional $\F_{\eps,hom}$ concentrates all of its energy here, so this suffices to establish the liminf-inequality. Such a restriction however is only possible if $W_\delta \geq - C\delta$ -- otherwise, counterexamples like in Example \ref{example voids} show that it may be possible to create {\em negative} energy on the domain where $u$ is close to pure phase $\pm 1$ unless $\delta$ is much smaller not just than $\eps$, but also $\eps^{3/2}$.

\begin{lemma}\label{lemma liminf}
Assume that $\Omega, W_\delta, W_{hom}$ satisfy the assumptions in Section \ref{section assumptions}. Assume that $\delta_\eps$ is a parametrized family such that $\lim_{\eps\to 0^+} \delta_\eps/\eps = 0$ and consider $\F_\eps := \F_{\eps, \delta_\eps}$ and $\F_{\eps,hom}$ as in \eqref{eq our modica mortola functionals}. Let $\Omega'\subset\Omega$ be open, compactly contained in $\Omega$ and $\gamma \in (0,1)$. If $u_\eps:\Omega\to\R$ satisfies $\liminf_{\eps\to 0^+} \F_\eps(u_\eps)<\infty$, there exists a family of function $\hat u_\eps:\Omega\to [-M, M]$ such that, for a subsequence $\eps\to 0$ which realizes the lower limit, we have
\[
\liminf_{\eps\to 0^+}\F_\eps(u_\eps) \geq \liminf_{\eps\to 0^+}\int_{\{-\gamma<\hat u_\eps < \gamma\}\cap \Omega'} \frac\eps2\,\|\nabla \hat u_\eps\|^2 + \frac{W_{hom}(\hat u_\eps)}\eps\dx, \qquad \lim_{\eps\to 0^+} \big\|\hat u_\eps - T_Mu_\eps \big\|_{L^p} = 0.
\]
\end{lemma}

\begin{proof}
For the moment, fix $\eps>0$ and denote $u = u_\eps$. Recall that $\F_{\eps,\delta}(T_Mv) \leq \F_{\eps,\delta}(v)$ as above. For $\tau>0$, we define
\[
\hat u \in \argmin_v \F_{\eps, \delta}(v) + \frac1{2\tau}\,\| v - T_Mu\|_{L^2(Q)}^2.
\]
A minimizer exist $\hat u \in H^1(\Omega)$ exists and satisfies $-M\leq \hat u \leq M$ since the functional is coercive, but without further assumptions, we do not know that $\hat u$ is unique. We note however that
\begin{enumerate}
\item $\F_{\eps, \delta_\eps}(\hat u) \leq \F_{\eps, \delta_\eps}(T_Mu) \leq \F_{\eps, \delta_\eps}(u)$.
\item $\| \hat u - T_Mu\|_{L^2} \leq \sqrt{2\tau\big(\F_{\eps}(u) - \F_{\eps}(\hat u)\big)}$. 
\item $T_M\hat u = \hat u$ since both terms in the energy $\F_\eps$ are non-increasing under truncation, and the first term is strictly decreasing unless $|\hat u|\leq M$ almost everywhere. 
\end{enumerate}

Taking $\tau = \tau_\eps\to 0$, we note that 
\[
\|\hat u-T_Mu\|_{L^p(\Omega)} = \left(\int_\Omega |\hat u-T_M u|^p\cdot 1\dx\right)^\frac1p\leq |\Omega|^{1- \frac p2} \|\hat u-T_m\|_{L^2(\Omega)}\lesssim \tau_\eps^{1/2}
\]
for $1\leq p<2$ by H\"older's inequality and 
\[
\|\hat u-T_Mu\|_{L^p(\Omega)} \leq \|\hat u - T_Mu\|_{L^2(\Omega)}^\frac2p \|\hat u-T_Mu\|_{L^\infty(\Omega)}^{1-\frac 2p} \lesssim C\tau_\eps^{1/p}\to 0
\]
for $2<p<\infty$ by H\"older's inequality/interpolation as in \cite[Exercise 4.4]{brezis2011functional}.
Due to the smoothness assumptions on $W_\delta$, we find that $\hat u$ satisfies the Euler-Lagrange equation
\[
-\eps \Delta \hat u + \frac1\eps\,(\partial_uW_\delta)(x, \hat u) + \frac{\hat u - T_Mu}\tau = 0.
\]

Let $\rho_\eps$ be any length scale such that $\delta_\eps \ll \rho_\eps \ll \eps$, for example the geometric mean $\rho_\eps = \sqrt{\eps\delta_\eps}$. For the ease of presentation, we abbreviate $\rho= \rho_\eps$ and $\delta = \delta_\eps$ while keeping the dependence on $\eps$ in mind.
For any $x_0 \in \Omega$, the rescaled function $\hat u_\rho (y) = \hat u(x_0+\rho y)$ satisfies
\begin{align*}
\Delta \hat u_\rho (y) &= \rho^2\,(\Delta \hat u) (x_0 +\rho y)\\
	&=  \frac{\rho^2}{\eps^2}\partial_u W_\delta\big(x_0+\rho y, \:\hat u_\rho (y)\big) + \frac{\rho^2}{\tau\eps} \,\big(\hat u - T_Mu\big)(x_0 +\rho y)
\end{align*}
on the set $\{ y : x_0 +\rho y\in \Omega\}$. In particular, if $B_{(R+1)\rho}(x_0)\subseteq\Omega$, then $\|\hat u_\rho\|_{L^\infty(B_{R+1})} \leq M$ and, for $p$ as used in Section \ref{section assumptions} to control the Lipschitz-constant of $W_\delta(x, \cdot)$, we have
\begin{align*}
\|\Delta \hat u_\eps\|_{L^p(B_{R+1})} &\leq \left( \sum_{\{i : Q_i^\delta \cap B_{(R+1)\rho}(x_0)\neq \emptyset\}} \int_{Q_i^\delta} \left(\frac{\rho^2}{\eps^2}\right)^p \,L_\delta(x)^p \rho^{-d} \dx  \right)^\frac1p + \frac{2M\rho^2}{\eps\tau}\,|B_{R+1}|^{1/p}\\
	&= \left(\frac\rho \eps\right)^2\left( \rho^{-d} \sum_{\{i : Q_i^\delta \cap B_{(R+1)\rho}(x_0)\neq \emptyset\}} |Q_i^\delta|\cdot \frac1{|Q_i^\delta|} \int_{Q_i^\delta} L_\delta(x)^p \dx  \right)^\frac1p + \frac{2M\rho^2}{\eps\tau}\,|B_{R+1}|^{1/p}\\
	&\leq \left(\frac\rho \eps\right)^2\left(C \rho^{-d} \sum_{\{i : Q_i^\delta \cap B_{(R+1)\rho}(x_0)\neq \emptyset\}} |Q_i^\delta|  \right)^\frac1p + \frac{2M\rho^2}{\eps\tau}\,|B_{R+1}|^{1/p}\\
	&\leq C \left(\frac\rho \eps\right)^2\,\rho^{-d/p}\left| \bigcup_{\{i : Q_i^\delta \cap B_{(R+1)\rho}(x_0)\neq \emptyset\}} Q_i^\delta\right|^\frac1p + \frac{2M\rho^2}{\eps\tau}\,|B_{R+1}|^{1/p}\\
	&\leq C (\rho/\eps)^2\, \rho^{d/p} \big| B_{(R+1)\rho + R\delta}(x_0)\big|+ (2M\rho^2)/{\eps\tau}\cdot |B_{R+1}|^{1/p}
\end{align*}
since the cells $Q_i^\delta$ are disjoint and their diameter is bounded above by $R\delta$. Since $R\delta\leq R\rho$, we find that
\[
\|\Delta \hat u_\eps\|_{L^p(B_2)}  \leq \big(C\,(\rho/\eps)^2 (2R+1)^{d/p} + 2^{d/p}\cdot 2M\rho^2/\tau\big)\,\big|B_1\big|^{1/p}.
\]
In particular, the right hand side is uniformly bounded (and in fact, decays to zero) if we select $\tau = \eps$ or similar.

By localizing \cite[Corollary 9.10]{gilbarg1977elliptic} after multiplying with a smooth cut-off function which is one on $B_1(0)$ and vanishes outside of $B_2(0)$, we find that $\|\hat u_\rho\|_{W^{2,p}(B_1)}\leq C$ for all $x_0\in \Omega$ such that the larger ball $B_{2\rho}(x_0)$ is contained in $\Omega$. Without loss of generality, we may assume that $p<d$. By the Sobolev embedding Theorem \cite[Satz 6.11]{dobrowolski2010angewandte}, we note that $W^{2,p}$ embeds into $W^{1, \frac{dp}{d-p}}$ which in turn embeds into $C^{0,\alpha}$ with $\alpha = 1- d\frac{d-p}{dp} =   2-\frac dp >0$ by Morrey's Theorem (see e.g.\ \cite[Satz 6.25]{dobrowolski2010angewandte}).

Rescaling to the original coordinates, we find that $|\hat u (x) - \hat u(y)| \leq C\,\big(|x-y|/\rho\big)^\alpha$ if $|x-y|<\rho$ and $B_{2\rho}(x)\subseteq\Omega$. We assume that $R\delta<\rho$, which is admissible since $\delta\ll \rho \ll \eps$ by assumption. Let us consider three index sets for the interface, the boundary, and the remainder of the domain:
\begin{align*}
I_{int} &:= \{i : 1\leq i \leq N_\delta \text{ s.t. }\exists\ x \in Q_i^\delta \text{ s.t.\ }B_{(R+1)\rho}(x)\subseteq \Omega\text{ and } \exists\ z\in Q_i^\delta \text{ s.t. }|u_\eps(z)|\leq \gamma\}\\
 I_{bd} &:=  \{i : 1\leq i \leq N_\delta \text{ s.t. there is no } x \in Q_i^\delta \text{ s.t. } B_{(R+1)\rho}(x)\subseteq \Omega\}\\
 I_{dom} &:=  \{1,\dots, N_\delta\}\setminus \big(I_{int}\cup I_{bd}\big).
\end{align*}
Then, by assumption we have
\[
\sum_{i\in I_{bd}\cup I_{dom}} \int_{Q_i^\delta}\frac\eps2\,\|\nabla \hat u\|^2 + \frac{W_\delta(x, \hat u)}\eps\dx \geq \frac{|\Omega|\,\inf_{x\in\Omega, \gamma \leq |u|\leq M} W_\delta(x, u)}\eps \gtrsim  -\frac\delta\eps.
\]
Furthermore, at the interface we note that if $i\in I_{int}$ and $\eps$ (and hence $\rho, \delta$) are small enough, then $|\hat u_\eps| \leq (1+\gamma)/2$ on $Q_i^\delta$ due to the uniform H\"older regularity bound on small scales. We conclude that 
\begin{align*}
\sum_{i\in I_{int}}& \int_{Q_i^\delta} \frac\eps 2 \,\|\nabla \hat u\|^2 + \frac{W_\delta(x, \hat u)}\eps \dx\\
	&\geq \int_{\bigcup_{i\in I_{int}} Q_i^\delta} \frac{(1-C_P\alpha)\eps} 2 \,\|\nabla \hat u\|^2 + \frac{W_{hom}(\hat u) }\eps\dx \\&\hspace{2cm}
	- C\,\left(\frac{\delta}{\eps}\right)^2\frac 1{\alpha\eps}\,\left|\bigcup_{i \in I_{int}} Q_i^\delta\right|
	 - \sum_{i=1}^{N_\delta}\left|\int_{Q_i^\delta} \frac{W_\delta(x,\langle \hat u\rangle) - W_{hom}(\langle \hat u\rangle)}\eps\dx\right|.
\end{align*}
Selecting $\alpha = \delta/\eps$, we note that 
\[
\int_{\bigcup_{i\in I_{int}}Q_i^\delta} \frac{W_{hom}(\hat u) }\eps\dx- C\,\left(\frac{\delta}{\eps}\right)\frac 1{\eps}\,\left|\bigcup_{i \in I_{int}} Q_i^\delta\right|
\geq \frac1\eps \left(\min_{|u| \leq (1+\gamma)/2} W(u) - C\,\frac\delta\eps\right) \left|\bigcup_{i\in I_{int}} Q_i^\delta\right|.
\]
If $\delta/\eps$ is small enough, we conclude that $\big|\cup_{i\in I_{int}} Q_i^\delta\big| \lesssim \eps$. Furthermore, for fixed $\Omega', \gamma$, we find that $\{-\gamma < \hat u_\eps < \gamma\}\cap \Omega' \subseteq \bigcup_{i\in I_{int}} Q_i^\delta$ and hence
\begin{align*}
\F_{\eps, \delta_\eps}(\hat u_\eps) &= \sum_{i=1}^{N_\delta} \int_{Q_i^\delta} \frac\eps2 \,\|\nabla \hat u_\eps\|^2 + \frac{W_\delta(x,\hat u)} \eps\dx\\
	&\geq \left(1- C_P\frac\delta\eps\right)\int_{\{-\gamma < \hat u_\eps < \gamma\}\cap \Omega'} \frac\eps2 \,\|\nabla \hat u_\eps\|^2 + \frac{W_{hom}(u)} \eps \dx - C\,\frac\delta\eps.\qedhere
\end{align*}
\end{proof}

\begin{remark}[On the boundary regularity of $W_\delta$]\label{remark boundary stuffs}
We note that 
\[
\bigcup_{i\in I_{bd}}Q_i^\delta \subseteq \big\{x : \dist(x, \partial\Omega) < (R+1)\rho\big\} \qquad \Ra\quad \left| \bigcup_{i\in I_{bd}}Q_i^\delta\right| \leq C\rho.
\]
If the upper Minkowski content of $\partial\Omega$ is bounded, i.e.\ if
\[
\limsup_{\delta\to 0^+} \big|B_\delta(\partial\Omega)\big| := \limsup_{\delta\to 0^+} \left| \{x \in\Omega : \dist(x, \partial\Omega)<\delta\}\right| <\infty,
\]
then it would suffice to assume that $W_\delta$ is bounded from below in a $\delta$-neighborhood of $\partial\Omega$. 
This is the case for all sets with sufficiently smooth boundary (for instance, Lipschitz sets). We could allow for greater flexibility with irregular potentials at regular boundaries. Additionally, we do not have to require shape regularity for $Q_i^\delta$ within distance $\sim\rho\geq \delta$ of the boundary.
\end{remark}

\begin{remark}[On the minimizing movements step]
We recall that $H^1(\Omega)$ does not embed into $L^\infty(\Omega)$ in dimension $d\geq 2$. For instance, the sequence of functions
\[
v_r(x) = \min\left\{1,  \frac{\log\|x\|}{\log r}\right\} \cdot 1_{\{\|x\|\leq 1\}}
\]
converges to $0$ in $H^1(\Omega)$ as $r\to 0$, but $u_r(0) = 1$ for all $r>0$. Using functions of this type, we could construct finite energy sequences for which ever cell $Q_i^\delta$ contains a point $x \in u_\eps^{-1}(0)$. The minimizing movements step is therefore required to improve $u_\eps$ to $\hat u_\eps$ where a statement like Lemma \ref{lemma liminf} can be proved.
\end{remark}

\subsection{The main result}
Lemma \ref{lemma liminf} suffices to prove the critical $\Gamma-\liminf$ inequality and thus the main result.

\begin{theorem}\label{theorem gamma-convergence}\label{theorem main}
Assume that the conditions of Section \ref{section assumptions} are met.

\begin{enumerate}
\item If $\liminf_{\eps \to 0^+} \F_\eps(u_\eps) < +\infty$, there exists a subsequence of $u_\eps$ and a function $u\in BV(\Omega; \{-1,1\})$ such that $u_\eps \xrightarrow{d_M} u$ and 

\item With respect to the notion of $d_M$ convergence, we have
\[
\Gamma-\lim_{\eps\to 0^+} \F_\eps = \Gamma-\lim_{\eps\to 0^+} \F_{\eps, hom} = c_{hom}\cdot \Per_\Omega
\]
where $c_{hom} = \int_{-1}^1\sqrt{2\,W_{hom}(u)}\d u$.
\end{enumerate}
\end{theorem}

\begin{proof}
{\bf Step 1. Compactness.} Assume that $\liminf_{\eps\to \infty} \F_\eps(u_\eps)$ is finite. Select $\hat u_\eps \in H^1(\Omega; [-M, M])$ as in Lemma \ref{lemma liminf}. Fix $\gamma\in (0,1)$ for now. We note that
\begin{align*}
\liminf_{\eps\to 0^+} \F_\eps (u_\eps) &\geq \int_{\{-\gamma < \hat u_\eps < \gamma\}} \frac\eps2\,\|\nabla \hat u_\eps\|^2 + \frac{W_{hom}(\hat u_\eps)}\eps\dx\\
	&\geq \int_{\{-\gamma < \hat u_\eps < \gamma\}} \sqrt{2\,W_{hom}(\hat u_\eps)}\,\|\nabla u_\eps\|\dx\\
	&= \int_\Omega \sqrt{2\,W_{hom}(T_\gamma \hat u_\eps)}\,\|\nabla T_\gamma u_\eps\|\dx
\end{align*}
since $\nabla T_\gamma u_\eps=0$ whenever $W_{hom}(T_\gamma \hat u_\eps) \neq W_{hom}(\hat u_\eps)$. We denote by $G_{hom}$ an anti-derivative of $\sqrt{2W_{hom}}$ and conclude that
\[
\int_\Omega \|\nabla G_{hom}(T_\gamma \hat u_\eps)\|\dx \leq \liminf_{\eps\to 0^+} \F_\eps (u_\eps)  < +\infty.
\]
Since $T_\gamma \hat u_\eps$ is bounded in $L^\infty(\Omega)$, so is $G_{hom}(T_\gamma \hat u_\eps)$ and thus it is bounded in $BV(\Omega)$. We use the compact embedding of $BV$ into the Lebesgue space $L^1$ \cite[Theorem 10.1.4]{attouch2014variational} to extract an $L^1$-convergent subsequence of $G_{hom}(T_\gamma\hat u_\eps)$. The limit point is a function of bounded variation by \cite[Theorem 10.1.1]{attouch2014variational}. The convergence of $\hat u_\eps$ follows immediately from the invertibility of $G_{hom}$ (since $G_{hom}' = \sqrt{2\,W_{hom}}>0$) and  boundedness. Since the subsequence is bounded in $L^\infty$ and converges in $L^1$, it converges in $L^p$ for all $p<\infty$ as in Lemma \ref{lemma liminf}.

Note that, if $\gamma' < \gamma$ and $T_\gamma u_\eps \to u^\gamma$ as $\eps\to 0$, then $T_{\gamma'}u_\eps = T_{\gamma'}T_\gamma u_\eps \to T_{\gamma'}u^\gamma$. We further note that by BV-compactness, also the limits of the truncated sequences have a limiting object $u= \lim_{\gamma\nearrow 1} u^\gamma \in BV(\Omega; [-1,1])$. This allows us to identify a single limiting object. Taking a sequence $\gamma_n\nearrow 1$ and nested subsequences in $\eps$, we find that there exists a function $u:\Omega\to [-1,1]$ and a subsequence in $\eps$ such that $T_\gamma u_\eps \to T_\gamma u$ for all $\gamma>0$.

In the same way as we treated the interfacial region where $W_{hom}$ is bounded away from zero in Lemma \ref{lemma liminf}, we can see that $\int_{\{|\hat u_\eps|>1/\gamma\}} W(\hat u_\eps)\dx\to 0$. We can therefore argue that $\hat u_\eps = T_Mu_\eps\to u$ in $L^p(\Omega)$ for all $p<\infty$. Similarly, we see that $\liminf_{\eps\to 0^+} \frac1\eps \int_{\{W_{hom}(\hat u_\eps)>\alpha\}} W_{hom}(\hat u_\eps)\dx < +\infty$. We selected a subsequence realizing the lower limit, and since $L^p$-convergence implies convergence almost everywhere (for a subsequence), this guarantees that $u$ only takes values in the potential wells, i.e.\ $u\in BV(\Omega;\{-1,1\})$. 

We have seen that $\hat u_\eps$ has a convergent subsequence. By Lemma \ref{lemma liminf}, this means that also $T_Mu_\eps$ has an $L^p$-convergent subsequence for all $p<\infty$. Naturally, more cannot be said without stronger assumptions: We did not rule out the scenario that $W(x,u) = 0$ for all $x$ and $|u|\geq 1$, in which case there is essentially no control over $u_\eps$ outside $(-M, M) = (-1,1)$.

{\bf Step 2. Liminf-inequality.} Assume that $u_\eps\xrightarrow{d_M} u$. Then clearly also $\hat u_\eps \xrightarrow {d_M}u$. We would like to argue that
\[
\liminf_{\eps\to 0^+}\F_{\eps,\delta_\eps}(u_\eps) \geq \liminf_{\eps\to 0^+} \F_{\eps, hom}(\hat u_\eps) \geq \big(\Gamma-\lim_{\eps\to 0^+}\F_{\eps, hom}\big)(u),
\]
but we lack the control over the full functional $\F_{\eps,hom}$ by $\F_{\eps,\delta_\eps}$ in Lemma \ref{lemma liminf}. However, we know that the energy of a low energy sequence for the Modica-Mortola functional localizes at the interface. As in the proof of compactness, we can argue that for fixed $\gamma\in(0,1)$ we have
\begin{align*}
\liminf_{\eps\to 0^+}\F_{\eps,\delta_\eps}(u_\eps) &\geq \liminf_{\eps\to 0^+} \int_\Omega \|\nabla G_{hom}(T_\gamma u_\eps)\| \dx \geq \int_\Omega \|\nabla G_{hom}(T_\gamma u)\|\dx
\end{align*}
by \cite[Theorem 10.1.1]{attouch2014variational}, where the integral on the right is interpreted in the generalized sense. Since $u$ only takes two values $\pm 1$, we see that $G_{hom}(T_\gamma u)$ only takes values $G_{hom}(\pm\gamma)$ and the functions are merely shifts in the $y$ direction:
\[
\int_\Omega \|\nabla G(T_\gamma u)\|\dx = \frac{G_{hom}(\gamma) - G_{hom}(-\gamma)}2\int_\Omega \|\nabla u\|\dx = c_{hom,\gamma}\,\Per_\Omega(\{u=1\})
\]
where $c_{hom,\gamma} = G_{hom}(\gamma) - G_{hom}(-\gamma)$ such that $c_{hom} = c_{hom,1}$. We use the fact that the jump of $u$ has height $2$, i.e.\ $\Per(\{u=1\}) = \frac12\int_\Omega\|\nabla u\|\dx$. Since the result holds for all $\gamma$, we can now take $\gamma\nearrow 1$ to see that $\liminf_{\eps\to 0^+}\F_{\eps,\delta_\eps}(u_\eps) \geq c_{hom}\,\Per_\Omega(\{u=1\})$.

{\bf Step 3. Limsup-inequality.} Let $\phi:\R\to\R$ be a solution of the ODE $\phi' = \sqrt{2\,W_{hom}(\phi)}$ such that $\phi(0) = 0$. The solution is unique on the set $\phi^{-1}(\R\setminus\{0\})$ since the square root function is locally  Lipschitz-continuous except at zero, but multiple solutions may exist if $\phi$ reaches $\pm 1$ in finite space. By design, also the truncation $T_1\phi$ is a solution, so we may assume that $\phi (x) \in [-1,1]$ for all $x$. We immediately conclude that $\phi$ is monotone increasing and Lipschitz-continuous since $0 \leq \phi' \leq \max_{u\in [-1,1]} W(u)$. Due to monotonicity, the limits $\lim_{x\to \pm\infty}\phi(x)$ exist and by contradiction, we see that $\lim_{x\to\pm \infty} \phi(x) = \pm 1$. 

By \cite[Theorem 1.24]{giusti1984functions} and the lower semi-continuity of the perimeter functional under $L^1$-convergence (i.e., lower semi-continuity of the norm under weak convergence), it suffices to prove the $\limsup$-inequality for functions $u = 1_E - 1_{\Omega\setminus E}$ where $E$ is a set with $C^2$-boundary $\partial E$. For such a set, we note that the signed distance function $sd_E(x) = \dist(x, \Omega\setminus E) - \dist(x, E)$ is Lipschitz-continuous on $\R^d$ and $C^2$-smooth in a neighborhood of $\partial E\cap \Omega$, where we have $\|\nabla sd_E\|\equiv 1$ -- see e.g. \cite[Chapter 14.6]{gilbarg1977elliptic}.

We fix a small $\eta>0$ and define $I= \{1,\dots, N_\delta\}$,
\[
u^\eps(x) = \phi\big(sd_E(x)/\eps\big)\qquad\text{and } I_{int} = \{i \in I : \exists\ x\in Q_i^\delta \text{ s.t. } W(u^\eps(x)) > \eta\}.
\]
For any $i\in I$ and any fixed point $x_i \in Q_i^\delta$, we note that
\begin{align*}
\int_{Q_i^\delta} W_\delta(x,u^\eps)\dx &= \int_{Q_i^\delta} W_\delta(x, u^\eps(x)) - W_\delta(x, u^\eps(x_i))  + W_\delta(x, u^\eps(x_i)) \\
	&\qquad \qquad - W_{hom}(u^\eps(x_i)) + W_{hom}(u^\eps(x_i)) - W_{hom}(u^\eps(x)) + W_{hom}(u^\eps(x))  \dx\\
	&\leq \int_{Q_i^\delta} \big(L_\delta(x) + L\big) \big|u^\eps(x) - u^\eps(x_i)\big| \dx + \int_{Q_i^\delta} W_{hom}(u^\eps(x))\dx\\
	&\qquad\qquad + \left| \int_{Q_i^\delta} W_\delta(x, u^\eps(x_i)) - W_{hom}(u^\eps(x_i))\dx\right|.
\end{align*}
 If $i\in I_{int}$, we use the fact that $u^\eps$ is $O(1/\eps)$-Lipschitz and that the diameter of $Q_i^\delta$ is proportional to $\delta$ to find that
\[
\int_{Q_i^\delta} \big(L_\delta(x) + L\big) \big|u^\eps(x) - u^\eps(x_i)\big| \dx \leq C\,|Q_i^\delta|\,\frac\delta\eps.
\]
Clearly $W(u^\eps(x))>0$ if and only if $sd_E(x) < c\eps$ where the constant $c$ depends on $\eta$, but not $\eps$. As previously, we note that $\sum_{i\in I_{int}} |Q_i^\delta|\leq C\eps$ for some $C>0$, so
\[
\sum_{i\in I_{int}} \int_{Q_i^\delta} \big(L_\delta(x) + L\big) \big|u^\eps(x) - u^\eps(x_i)\big| \dx = O\left(\frac\delta\eps\right).
\]
Away from the interface, the function $u^\eps$ has much smaller variation than the Lipschitz estimate suggests. If $x, x'\in Q_i^\delta$ for the same $\delta>0$ and $i\notin I_{int}$, then $sd_E(x)$ and $sd_E(x')$ have the same sign. Without loss of generality, we assume that both are positive. We find that
\[
|u^\eps(x) - u^\eps(x')| = \left|\phi\big(sd_E(x)/\eps\big) - \phi\big(sd_E(x')/\eps\big)\right| \leq \|\phi'\|_{L^\infty(\min\{sd_E(x), sd_E(x')\})} \,\|x-x'\|/\eps.
\]
Since $\sqrt{2W}$ is monotone in a neighborhood of the potential wells, so is $\phi'$ and the estimate becomes
\[
|u^\eps(x) - u^\eps(x')| \leq \phi'\left(\frac{\min\{sd_E(x), sd_E(x')\}}\eps\right)\,\frac{R\delta}\eps \leq \phi'\left(\frac{sd_E(x)}{2\eps}\right)
\]
if $\delta$ is small enough with respect to $\eps$. We can therefore estimate
\begin{align*}
\sum_{i\in I\setminus I_{int}} \int_{Q_i^\delta} \big(L_\delta(x) + L\big) \big|u^\eps(x) - u^\eps(x_i)\big| \dx 
	&\leq \frac{R\delta} \eps\int_{\Omega}\phi'\left(\frac{sd_E(x)}{2\eps}\right)\dx.
\end{align*}
To compute the integral, we change coordinates. Let $\nu_x$ be the interior unit normal to $\partial E$ and note that the map $\Psi :\partial E\times \R\to \R^d$, $\Psi(x, t) = x+ t\,\nu_x$ is surjective since $\partial E$ is a compact $C^2$-manifold (namely, if $z\in \partial E$ is such that $\|x-z\| = \dist(x, \partial E)$, then $x = \Psi(z, sd_E(x))$). Since $\Psi$ is $C^1$-smooth, its functional determinant is bounded on the compact set $\overline\Omega$. In particular, we have
\begin{align*}
\sum_{i\in I\setminus I_{int}} \bigg|\int_{Q_i^\delta} \frac{W_{hom}(u^\eps(x)) - W_{hom}(x, u^\eps(x))}\eps \dx \bigg|
	&\leq C\,\frac\delta\eps\,\int_{\partial E} \int_\R \phi'\left(\frac t\eps\right)\,\frac1\eps \dt \d H^{d-1}_x + O\left(\frac\delta\eps\right)\\
	&\qquad + \sup_{u\in \mathcal U_\delta}  \sum_{i=1}^{N_\delta} \left|\int_{Q^\delta_i} W_\delta(x, u_i) - W_{hom}(u_i)\dx \right|\\
	&= 2C\,\H^{d-1}(\partial E)\,\frac\delta\eps + O\left(\frac\delta\eps\right) = O\left(\frac\delta\eps\right)
\end{align*}
where $\H^{d-1}$ denotes the $d-1$-dimensional Hausdorff measure (the natural area measure on $\partial E$).
\end{proof}

\subsection{Extensions}
We did not enforce boundary conditions in the functionals $\F_\eps, \F_{\eps, hom}$. If we defined the functionals 
\begin{equation}\label{eq dirichlet boundary}
\F_{\eps,\delta}:L^1(\Omega)\to\R, \qquad \F_{\eps,\delta}(u) = \begin{cases} \int_\Omega \frac\eps2 \|\nabla u\|^2 + \frac{W_\delta(x,u)}\eps\dx &\text{if } u+1 \in H_0^1(\Omega)\\ +\infty &\text{else}\end{cases}
\end{equation}
instead, the result of Corollary \ref{theorem gamma-convergence} would still hold, but with the perimeter functional rather than the perimeter relative to $\Omega$:

\begin{corollary}[Dirichlet boundary conditions]
Assume that $\Omega, W_\delta, W_{hom}$ satisfy the conditions of Section \ref{section assumptions}. Let $\delta_\eps$ be a family of values such that $\lim_{\eps\to 0^+} (\delta_\eps/\eps) = 0$, and let $\F_\eps := \F_{\eps, \delta_\eps}$ as in \eqref{eq dirichlet boundary} (and $\F_{\eps, hom}$ analogously with homogeneous Dirichlet boundary condition). Then, with respect to $d_M$-convergence, we have
\[
\Gamma-\lim_{\eps\to 0^+} \F_\eps = \Gamma-\lim_{\eps\to 0^+} \F_{\eps, hom} = c_{hom}\cdot \Per.
\]
\end{corollary}

Many phenomena have multiple length scales. For instance, crystalline materials have `small' defects (inclusions, vacancies, individual dislocations) as well as larger defects such as grain boundaries in poly-crystalline materials. The same proof techniques used above apply if there are multiple small length scales. 

\begin{remark}[Multiple small length scales]\label{remark multiple length scales}
Assume that $\vec \delta = (\delta_1,\dots, \delta_n)$ is a collection of small length scales such that $\max_i \delta_i \ll \eps$. We assume that 

\begin{itemize}
\item the domain $\Omega$ can be partitioned into collections $\{Q_j^{\delta_i}\}_{j=1}^{N_{\delta_i}}$ for all $i$,
\item we are given potentials $W_{\delta_i}$ and $W_{hom,i}$ for all $i$, and
\item the necessary compatibility conditions of Section \ref{section assumptions} hold.
\end{itemize}

Then the functionals
\[
\F_{\eps, \vec \delta}(u) = \int_\Omega \frac\eps2\,\|\nabla u\|^2 + \frac1\eps \sum_{i=1}^n  W_{i}\left(\frac x{\delta_i}, u\right)\dx, \qquad 
\F_{\eps, hom}(u) = \int_\Omega \frac\eps2\,\|\nabla u\|^2 + \frac1\eps \sum_{i=1}^n  W_{hom,i}\left(u\right)\dx
\]
have the same $\Gamma$-limit. The same proof as above goes through, and we may require that $W_i\geq -C\delta_i$ for some $i$ and $\delta_i\ll \eps^{3/2}$ for others. This allows us to take care of various kinds of effects on several length scales.
 
As mentioned previously, this extension is relevant even in periodic homogenization with two length scales $\delta_1, \delta_2$ if $\delta_1/\delta_2\notin \mathbb Q$ since then there is no shared periodic reference configuration. However, even if $\delta_1/ \delta_2\in \mathbb Q$, the shared periodic domain may require a length-scale $\gg \max\{\delta_1, \delta_2\}$, possibly much larger than $\eps$. Thus, the flexibility of a `regularity'-driven approach is of great benefit in this setting.
\end{remark}

\begin{remark}[Multiple mixed length scales]
At no point did we make use of the fact that $W_{hom}$ does not depend on $x$. In Section \ref{section assumptions}, we could for instance allow for potentials $W_{\delta, \ell}$ with two length-scales and $W_{hom, \ell}$ where the shorter length-scale is no longer resolved, or potentials $W_{hom}(x,u)$ which vary in space on the macroscopic scale. For a simple statement of an extended result, we modify the assumptions as follows. 

\begin{enumerate}
\item [(1--3)] We retain the assumptions on $\Omega$ and the partition $Q_i^\delta$.
\setcounter{enumi}{3}

\item[(4--7)] We retain the assumptions on the double-well potential, but index it by two length-scales $W_{\delta, \ell}$ rather than just one.

\item [(8--10)] We consider a homogenized potential $W_{hom,\ell}$ which may still depend on $x$. We assume that $W_{hom, \ell}$ is measurable in $x$ and Lipschitz-continuous in $u$, uniformly in $x$, and that $W_{hom,\ell}(x, -1) = W_{hom,\ell}(x, 1) = 0 < W_{\hom,\ell}(x,u)$ for all $x\in\Omega$ and $u\neq \pm 1$. 

We add the following technical assumption: {\em Let $\ell_\eps$ be a parametrized family of length scales. The functionals
\[
\F_{\eps, \hom}(u) = \begin{cases} \int_\Omega \frac\eps2\,\|\nabla u\|^2 + \frac{W_{hom, \ell_\eps}(u)}\eps\dx & u\in H^1(\Omega)\\ +\infty &\text{else}\end{cases}
\]
converge to a limiting functional of the form
\[
\F(u) = \begin{cases}  \int_{\partial^*\{u=1\}} c(x, \nabla u/\|\nabla u\|)\,\d\H^{d-1} & \text{if }u\in BV(\Omega; \{-1,1\})\\ +\infty &\text{else}\end{cases}
\]
 in the topology of $d_M$-convergence where $c$ is a strictly positive function and $\partial^*E$ denotes the reduced boundary of an open set $E$. Additionally, for every $u \in BV(\Omega; \{-1,1\})$, there exists a sequence $u_\eps^*\in H^1(\Omega; [-M, M])$ such that 
\[
u_\eps^*\xrightarrow{L^1(\Omega)}u, \qquad \lim_{\eps\to 0^+} \F_{\eps,hom}(u_\eps^*) = \F(u)\qquad \text{and } |u_\eps^*(x) - u_\eps^*(y)|\leq \frac C\eps \,\|x-y\|
\]
for all $x,y\in \Omega$, i.e. there exists a sufficiently regular recovery sequence for $\F_{\eps,hom}$.
}

\item[(11)] We retain assumption (11), with the slight modification that $W_{hom}$ may also depend on $x$.
\end{enumerate} 

By the same proof as above, we conclude that the functionals
\[
\F_{\eps}(u) = \begin{cases} \int_\Omega \frac\eps2\,\|\nabla u\|^2 + \frac{W_{\delta_\eps, \ell_\eps}(u)}\eps\dx & u\in H^1(\Omega)\\ +\infty &\text{else}\end{cases}
\]
satisfy
\[
\Gamma(d_M)-\lim_{\eps\to 0^+}\F_{\eps} = 
\Gamma(d_M)-\lim_{\eps\to 0^+}\F_{\eps, hom} = \F(u).
\]
This allows us to consider, for instance, $W_{hom, \ell_\eps}(x,u) = \widetilde W(x/\ell_\eps, u)$ with length scales $\eps\ll \ell_\eps$, or a single fixed potential which varies spatially on a slow macroscopic scale, i.e.\ for which $\ell_\eps\equiv 1/\diam(\Omega)$. Generically, the $\Gamma$-limit is not the Euclidean perimeter functional in either situation.
\end{remark}

\subsection{On scale separation and potentials which take negative values}
Let us briefly discuss a further extension which arises when we modify the energy functional.
We have seen in Example \ref{example voids} and the proof of Lemma \ref{lemma eps32} that the required scale separation $\delta\ll \eps^{3/2}$ if $W_\delta$ is allowed to become negative comes from the fact that the gradient term -- the square of a norm -- must control small oscillations measured essentially by the $L^1$-norm. This scaling incompatibility leads us to introducing a modified functional which does not suffer from the same deficiency:
\[
\widetilde \F_{\eps,\delta} (u) = \int_\Omega \frac\eps2\,\|\nabla u\|^2 + \frac{W_\delta(x,u)}\eps + \sqrt{\frac\delta\eps} \,\|\nabla u\|\dx.
\]
Naturally, the gradient term vanishes in the $\Gamma$-limsup inequality for the optimal profile, but we will see that it equips us with sufficient regularity on small scales. 
If $W_\delta(x,u)$ is $L$-Lipschitz-continuous in $u$ for all fixed $x$ with the same constant $L$, then 
\begin{align*}
\int_{Q_i^\delta} \frac\eps2& \,\|\nabla u\|^2 + \frac{W_\delta(x,u)}\eps + \sqrt{\frac\delta\eps}\,\|\nabla u\|\dx\\
	&\geq \int_{Q_i^\delta} \frac{W_\delta(x, \langle u\rangle)- L|u- \langle u\rangle|}\eps + \sqrt{\frac\delta\eps}\,\|\nabla u\|\dx\\
	&\geq \int_{Q_i^\delta} \frac{W_{hom}(u)}\eps \dx - \left|\int_{Q_i^\delta}\frac{W_\delta(x, \langle u\rangle)-W_{hom}(u) }\eps \dx\right| +\frac1\eps \int_{Q_i^\delta} \sqrt{\delta\eps}\|\nabla u\| -  L\,\big|u-\langle u\rangle\big| \dx
\end{align*}
The first term is non-negative and the second is under control by the compatibility between $W_\delta$ and $W_{hom}$. The third can be controlled if we assume that a Poincar\'e inequality holds for all $Q_i^\delta$ with the natural scaling at the length-scale $\delta$:
\begin{equation}\label{eq l1 poincare}
\int_{Q_i^\delta} |u-\langle u\rangle|\dx \leq C_P^1 \, \delta\int_{Q_i^\delta} \|\nabla u\|\dx.
\end{equation}
If $\eps/\delta$ is large enough, the third term is therefore non-negative. By the same proof as above, we can show the following.

\begin{corollary}\label{corollary modified modica mortola}
Assume that $Q_i^\delta, W_\delta, W_{hom}$ satisfy the conditions specified in Section \ref{section assumptions}, except for Condition \eqref{condition scale separation}, which we replace by assuming that 
\begin{itemize}
\item $W_\delta$ is bounded from below and
\item There exists $C_P^1>0$ such that \eqref{eq l1 poincare} holds for all $\delta>0$ and all $i=1,\dots, N_\delta$.
\end{itemize}
Furthermore, strengthen Condition \eqref{condition lipschitz} so that $W_\delta$ is Lipschitz-continuous in $u$ with Lipschitz-constant $L$ for all $x$.
Let $\delta_\eps \ll \eps$. Then
\[
\Gamma(L^1)-\lim_{\eps\to 0^+}\widetilde \F_{\eps,\delta_\eps} = \Gamma(L^1) -\lim_{\eps\to 0^+}\F_{\eps, hom}.
\]
\end{corollary}

We note that the inclusion of a total variation term in $\widetilde \F_\eps$ negates one of the functional's main advantages over other models: The fact that the energy is quadratic in the highest order derivatives, leading to a semi-linear evolution equation in the gradient flow. In the numerical experiments of Section \ref{section numerical}, we do observe that the discrepancy in the perimeter functional for varying $\delta$ is much more pronounced for potentials $W_\delta$ with shifting wells which become negative than in other settings under consideration.

When $W$ displays more `quadratic' behavior where it becomes negative, no modification to the functional as there is no mismatch between a square norm and a norm.

\begin{remark}[An unphysical example]
We note that the constraint $W_\delta \geq 0$ is sufficient for convergence even if $\eps^{3/2}\ll \delta \ll \eps$, but not necessary.
Let us revisit the setting of Example \ref{example voids} for an example which does not have physical applications to the best of our knowledge, but which demonstrates the extraordinary resilience of the method to various kinds of perturbations for $\delta\ll \eps$. Namely, we are looking at a potential which, on parts of the domain, is pushing $u$ away from the potential wells.

Specifically, consider $W_\delta(x,u) = a(x/\delta)\,\max\{1-u^2, 0\}^2$ where $a$ is a $1$-periodic function on $\R^d$ which changes sign, but has a positive integral. In particular,
\[
W_{hom}(u) = \left(\int_{(0,1)^d} a(x)\dx\right) \,\max\{1-u^2, 0\}^2\geq 0.
\]
We claim that
\[
\int_{Q_i^\delta} \frac\eps2\,\|\nabla u\|^2 + \frac{W_\delta(x, u)}\eps\dx \geq 0
\]
for all cubes $Q_i^\delta$ of side-length $\delta$ with edges parallel to the coordinate axes (at least away from the boundary). By rescaling, the minimization problem is equivalent to the claim that 
\[
\int_{(0,1)^d} \frac{\eps/\delta}2\,\|\nabla u\|^2 + \frac{W(x, u)}{\eps/\delta}\dx \geq 0 \qquad \forall\ u \in H^1\big((0,1)^d\big).
\]
Following the proof of Lemma \ref{lemma liminf}, we only need to establish this if $u$ is $\eps$-Lipschitz continuous on a set of diameter $\delta$, so $u$ takes values close to a single value $\bar u$ (at least away from the boundary). As noted in Remark \ref{remark boundary stuffs}, the boundary can easily be treated separately. Naturally, the expression is positive for $u$ which are uniformly close to $\bar u \not\approx 1$, so we consider the simpler problem of approximating $W$ by its second order Taylor expansion around $1$, i.e.\ we consider
\[
\int_{(0,1)^d} \frac{(\eps/\delta)^2}2\,\|\nabla v\|^2 + 4a(x)\,v^2\dx
\]
where $v = u-1$. This is admissible since to minimize the energy, we can always consider the truncated function $T_1u$ and assume that $u$ only takes values in $[-1,1]$. Write $a(x) = 1 - b(x)$ and note that for $d>2$ we have
\[
\int b\,u^2\dx \leq \left(\int_{(0,1)^d} |b|^\frac d2\dx \right)^\frac2d\left(\int |u|^\frac{2d}{d-2}\dx\right)^\frac{d-2}d \leq C\,\|b\|_{L^{d/2}} \|u\|_{H^1((0,1)^d)}^2
\]
since $H^1$ embeds continuously into $L^{d/2}$. In particular
\begin{align*}
\int_{(0,1)^d} &\frac{(\eps/\delta)^2}2\,\|\nabla v\|^2 + 4a(x)\,v^2\dx\\
	& \geq \int_{(0,1)^d} \frac{(\eps/\delta)^2 - C\,\|b\|_{L^{d/2}((0,1)^d)}}2\,\|\nabla v\|^2 + \big(4 - C\,\|b\|_{L^{d/2}((0,1)^d)}\big)\,v^2\dx\geq 0
\end{align*}
if $\|b\|_{L^{d/2}}$ is small enough. The proof of $\Gamma$-convergence goes through for this class of potentials as before. In dimension $d=2$, the same argument holds using $\|b\|_{L^p}$ for any $p>1$.
\end{remark}

\section{Numerical Illustration}\label{section numerical}

\begin{figure}
\includegraphics[clip = true, trim = 1cm .5cm 1cm 1cm, height=3.2cm]{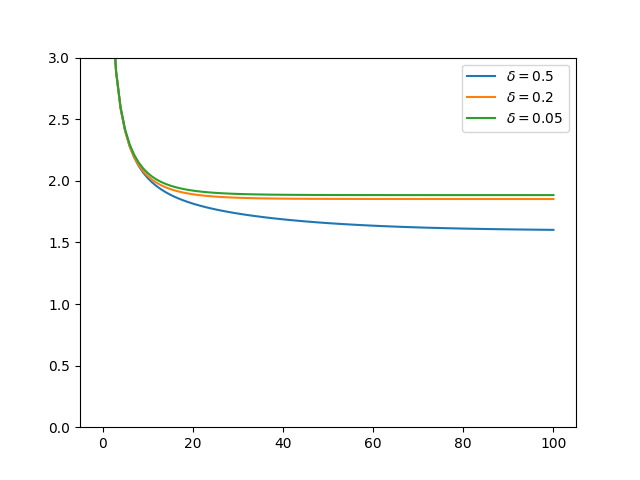}\hfill
\includegraphics[clip = true, trim = 2.4cm 1cm 3.7cm 1cm, height=3.2cm]{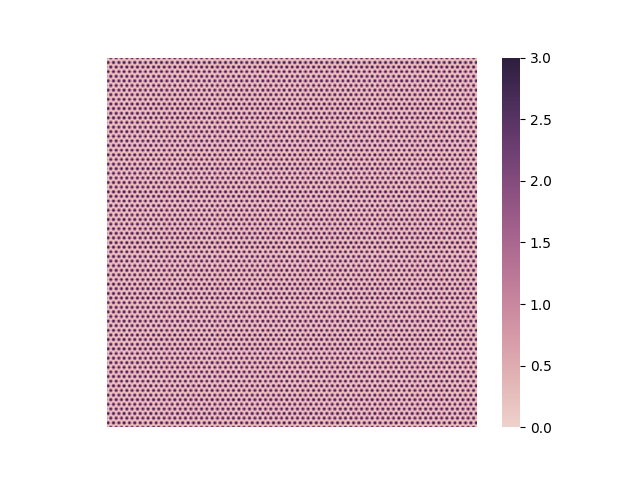}
\includegraphics[clip =  = true, trim = 2.4cm 1cm 3.7cm 1cm, height=3.2cm]{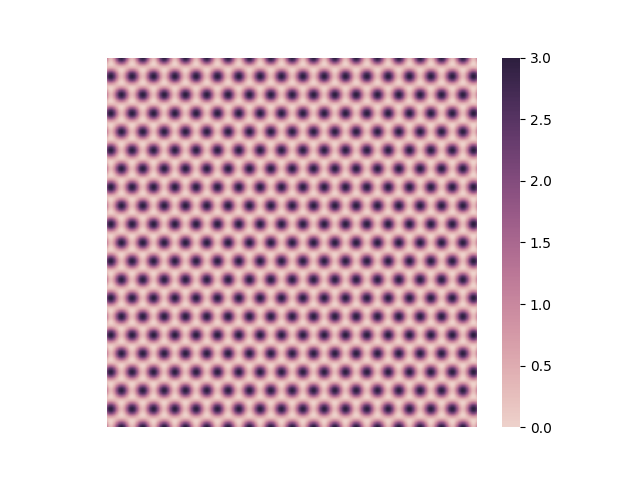}
\includegraphics[clip =  = true, trim = 2.4cm 1cm 1.5cm 1cm, height=3.2cm]{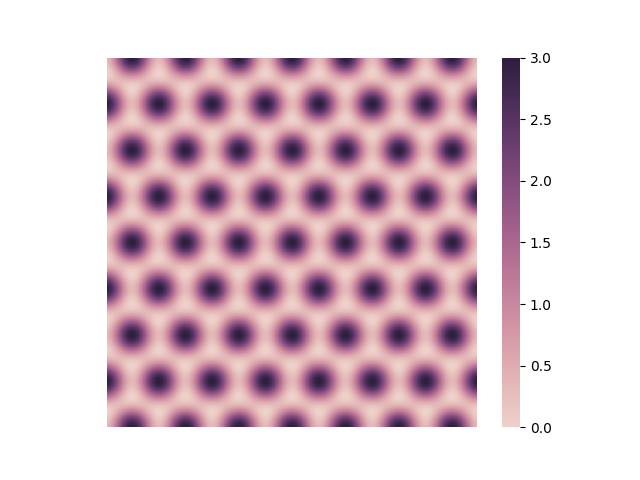}

\hfill\includegraphics[clip = true, trim = 2.4cm 1cm 3.7cm 1cm, height=3.2cm]{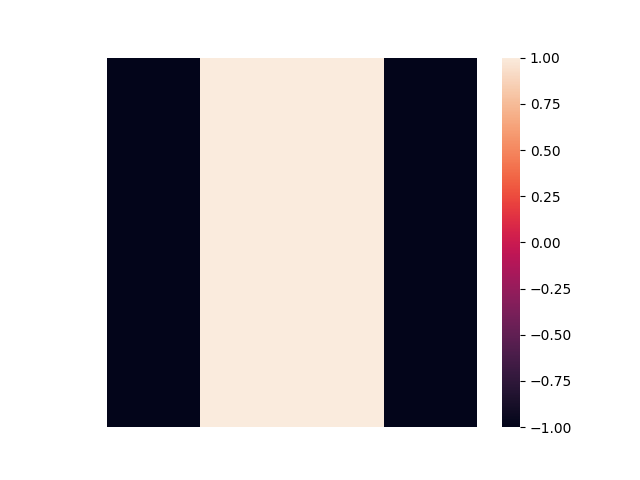}\hfill
\includegraphics[clip = true, trim = 2.4cm 1cm 3.7cm 1cm, height=3.2cm]{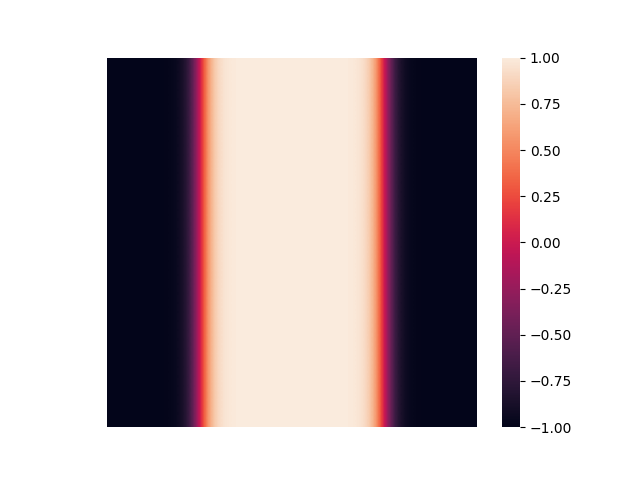}
\includegraphics[clip = true, trim = 2.4cm 1cm 3.7cm 1cm, height=3.2cm]{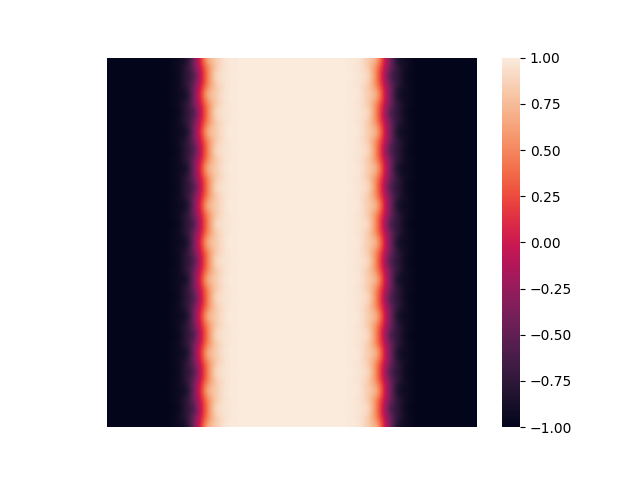}
\includegraphics[clip = true, trim = 2.4cm 1cm 1.5cm 1cm, height=3.2cm]{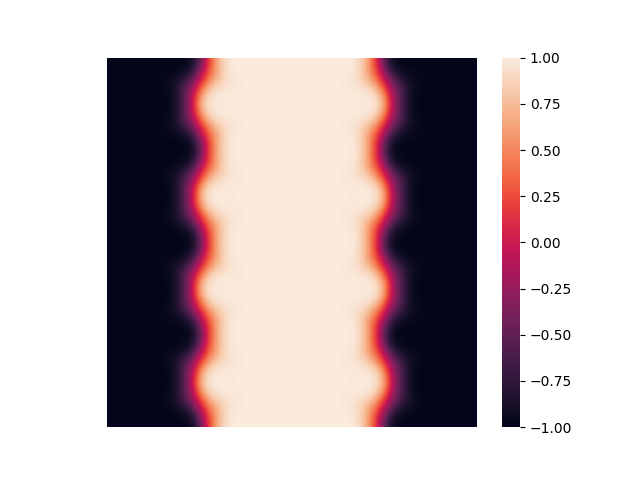}

\caption{\label{figure hexagon symmetry}
Top row: The decay of the spatially inhomogeneous `Modica-Mortola' energy (left) and the spatial factor $w(x/\delta)$ in the double-well potentials for different values of $\delta$ in a hexagonal symmetry setting. Bottom row: The initial condition (left) and the terminal state of our simulations for various values of $\delta$ (corresponds to to top row). We observe that if $\delta$ is large with respect to $\eps$, the interface adapts to the spatial microstructure. If $\eps$ is large with respect to $\delta$, on the other hand, the interfaces remain straight. 
}
\end{figure}

\begin{figure}
\includegraphics[clip = true, trim = 1cm .5cm 1cm 1cm, height=3.2cm]{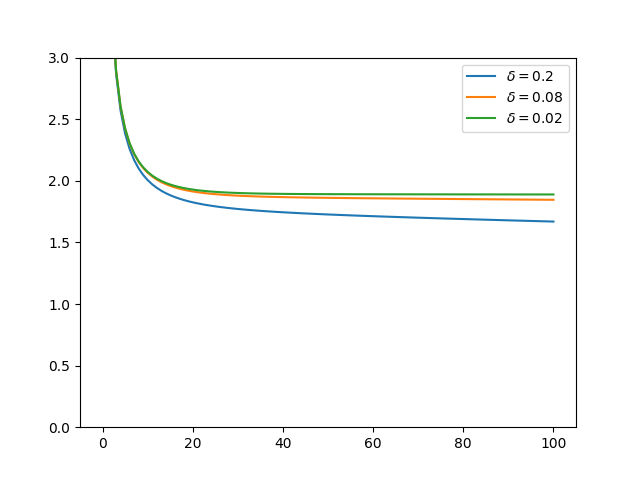}\hfill
\includegraphics[clip = true, trim = 2.4cm 1cm 3.7cm 1cm, height=3.2cm]{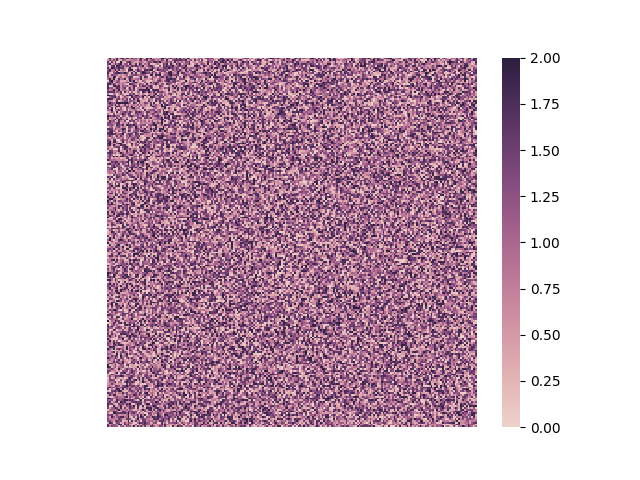}
\includegraphics[clip =  = true, trim = 2.4cm 1cm 3.7cm 1cm, height=3.2cm]{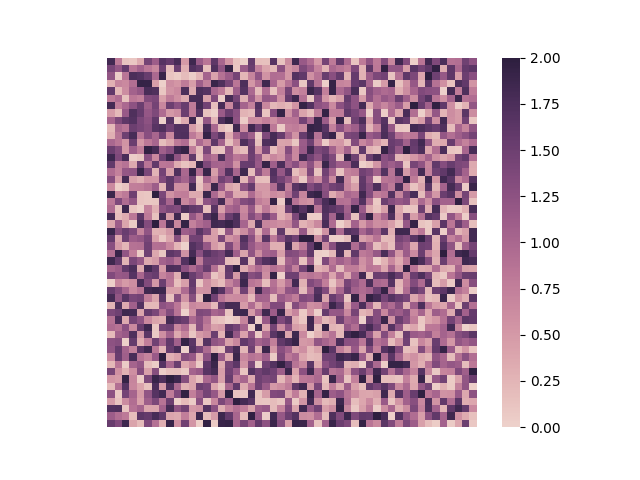}
\includegraphics[clip =  = true, trim = 2.4cm 1cm 1.5cm 1cm, height=3.2cm]{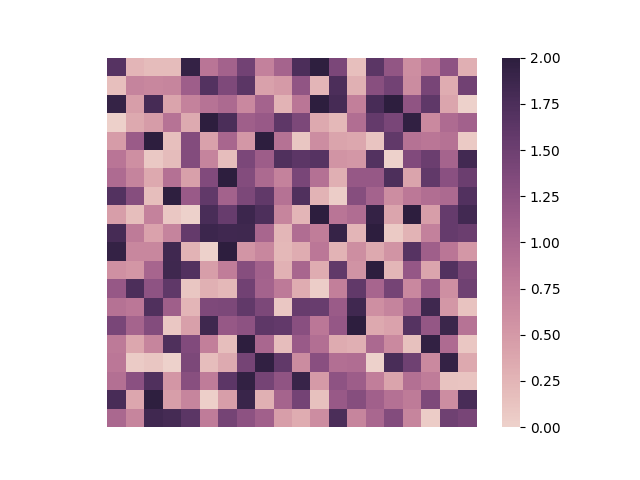}

\hfill\includegraphics[clip = true, trim = 2.4cm 1cm 3.7cm 1cm, height=3.2cm]{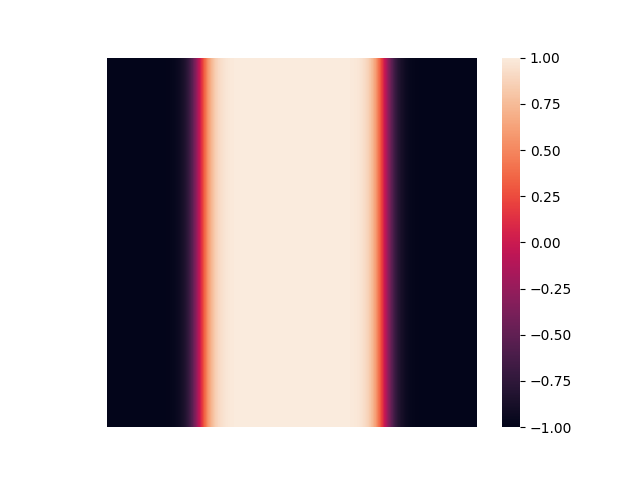}\hfill
\includegraphics[clip = true, trim = 2.4cm 1cm 3.7cm 1cm, height=3.2cm]{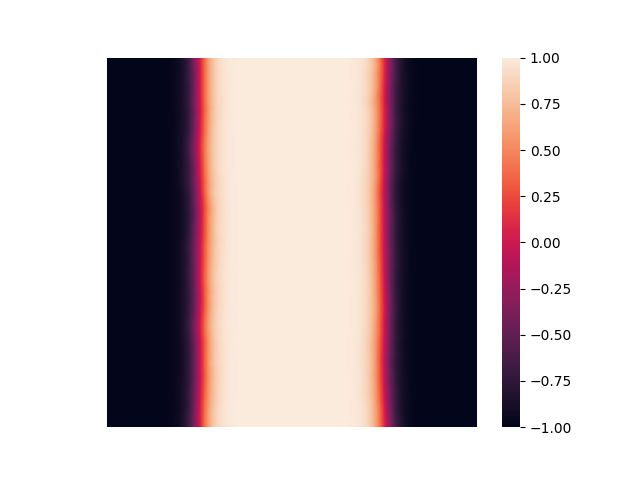}
\includegraphics[clip = true, trim = 2.4cm 1cm 3.7cm 1cm, height=3.2cm]{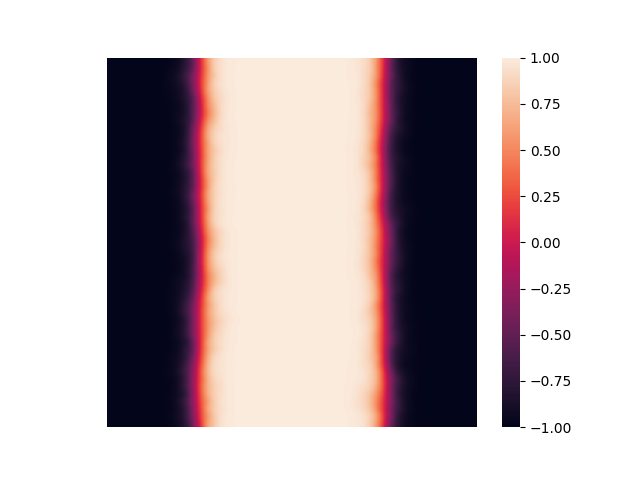}
\includegraphics[clip = true, trim = 2.4cm 1cm 1.5cm 1cm, height=3.2cm]{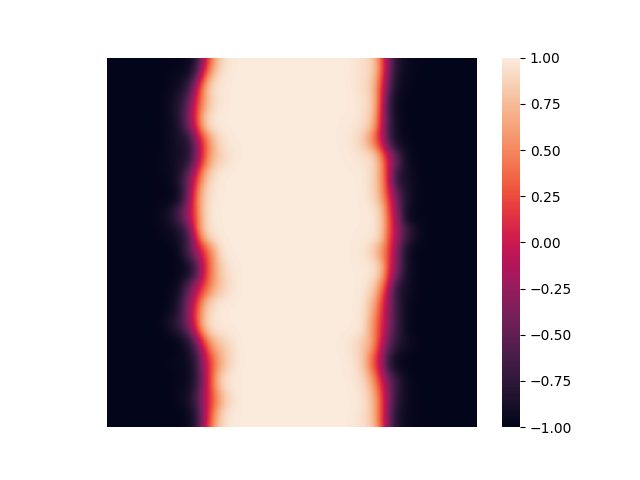}

\caption{\label{figure stochastic}
Top row: The decay of the spatially inhomogeneous `Modica-Mortola' energy (left) and the spatial factor $\sum_j q_{ij}^\delta 1_{Q_{ij}^\delta}$ in the double-well potentials for different values of $\delta$ in a random setting. Bottom row: The terminal state in the same simulation for the homogenized potential (left) and the terminal state of our simulations for various values of $\delta$ (corresponds to to top row).
}
\end{figure}

\begin{figure}
\includegraphics[clip = true, trim = 1cm .5cm 1cm 1cm, height=3.2cm]{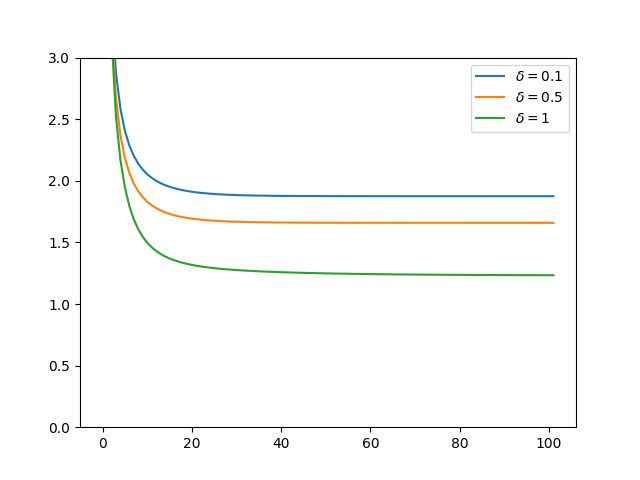}\hfill
\includegraphics[clip = true, trim = 2.4cm 1cm 3.7cm 1cm, height=3.2cm]{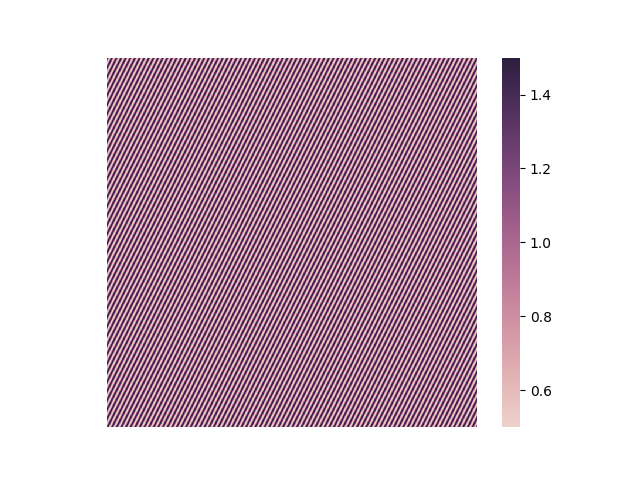}
\includegraphics[clip =  = true, trim = 2.4cm 1cm 3.7cm 1cm, height=3.2cm]{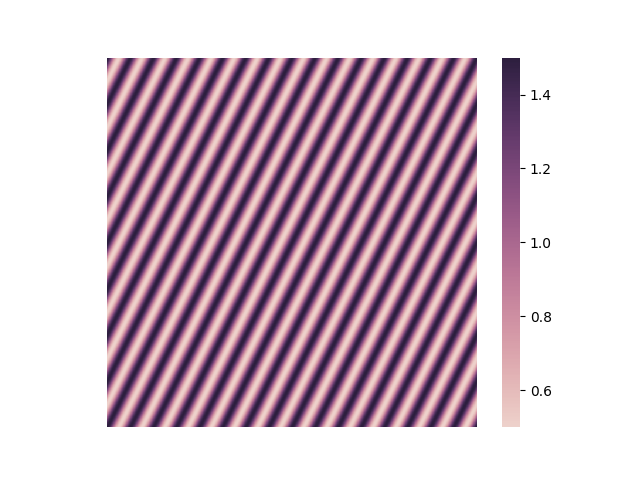}
\includegraphics[clip =  = true, trim = 2.4cm 1cm 1.5cm 1cm, height=3.2cm]{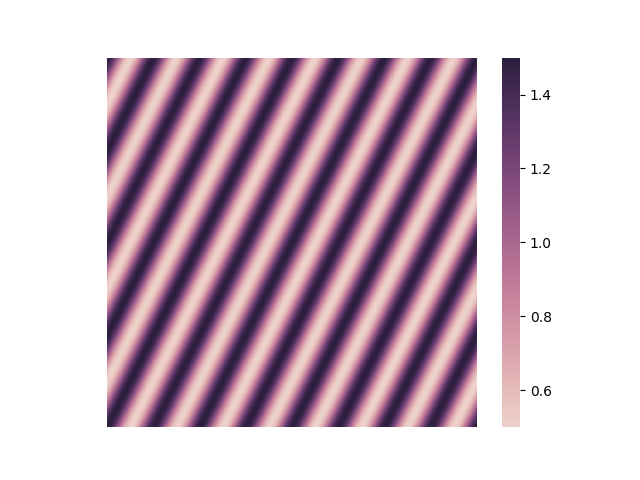}

\includegraphics[clip = true, trim = 5mm 5mm 1cm 5mm, height=3.2cm]{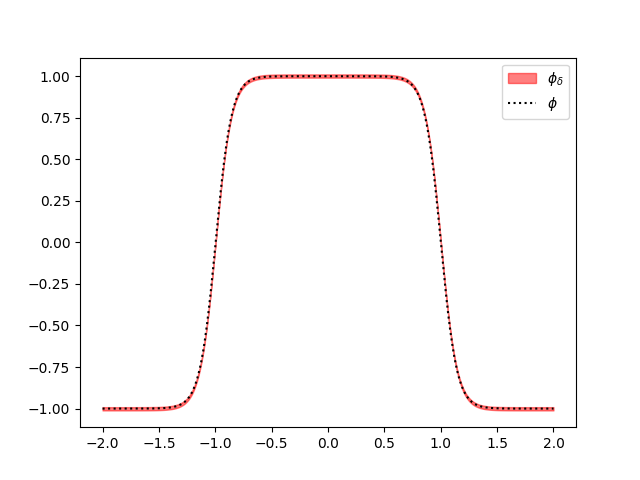}\hfill
\includegraphics[clip = true, trim = 2.4cm 1cm 3.7cm 1cm, height=3.2cm]{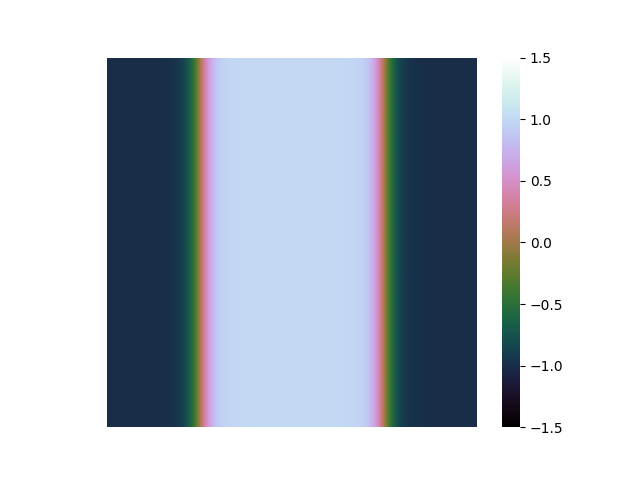}
\includegraphics[clip = true, trim = 2.4cm 1cm 3.7cm 1cm, height=3.2cm]{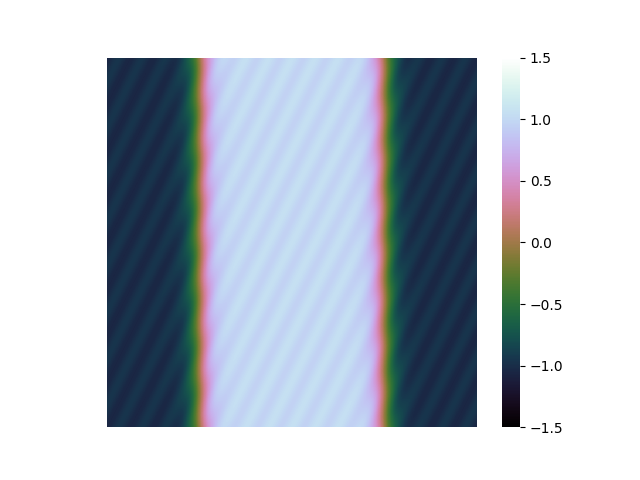}
\includegraphics[clip = true, trim = 2.4cm 1cm 1.5cm 1cm, height=3.2cm]{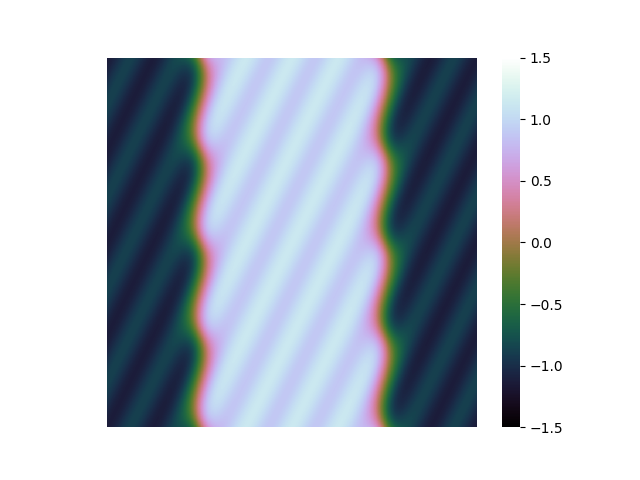}

\caption{\label{figure varying wells}
Top row: The decay of the spatially inhomogeneous `Modica-Mortola' energy (left) and the spatially varying wells at $\sqrt{b(x/\delta)}$ in the double-well potentials for different values of $\delta$ (plotting $b$). Bottom row: The left image shows the area between $\min_y u(x,y)$ and $\max_yu(x,y)$ for $\delta = 0.2$ (shaded red) compared to the solution for $W_{hom}$ (dotted). Right three images: The terminal state of our simulations for various values of $\delta$ (corresponds to to top row). We selected a different color palette in this figure to emphasize that $u$ takes values outside of $(-1,1)$ here.\\
Note that in agreement with our analysis, the discrepancy between terminal state energies for moderately small $\delta$ is much bigger in this setting than for potentials which are non-negative.
}
\end{figure}

{\bf Experiment design.} We numerically approximate the gradient flow of the phase-field energy 
\begin{equation}\label{eq energy numerical}
\F_{\eps,\delta}(u) = \avint_\Omega \frac\eps2 \,\|\nabla u\|^2 + \frac{W_\delta(x,u)}\eps \dx
\end{equation}
with $\eps = 0.025$ for various values of $\delta$ and various different potentials $W_\delta$:

\begin{enumerate}
\item In Figure \ref{figure hexagon symmetry}, we consider a periodic potential with hexagonal symmetry:
\[
W_\delta(x,u) = w(x/\delta)\,\frac{(u^2-1)^2}4, \qquad \widetilde w(x) = a + b \sum_{i=1}^3 \sin(\pi \cdot z_i)^2
\]
where
\[
a = 0.228, \qquad b = -0.1, \qquad z_1 = x_1, \qquad z_2 = \frac{x_1 + \sqrt{3}\,x_2}2, \qquad z_3 = \frac{x_1 - \sqrt{3}\,x_2}2
\]
and $w= \tilde w/\langle \tilde w\rangle$ is normalized such that the homogenized potential is the usual double-well potential $W_{hom}(u) = (u^2-1)^2/4$.

\item In Figure \ref{figure stochastic}, we consider a random potential:
\[
W_\delta(x,u) = \frac{(u^2-1)^2}4\sum_{i,j} q^\delta_{ij}\,1_{x\in Q_{ij}^\delta}
\]
where the $Q_{ij}^\delta$ form a tesselation of the spatial domain by squares of side-length $\delta$ and the coefficients $q^\delta_{ij}$ are independent uniformly distributed random variables in the interval $[0,2]$. 

\item In Figure \ref{figure varying wells}, we select a potential with spatially varying wells
\[
W_\delta(x,u) = \frac{\big(u^2- b(x/\delta)\big)^2 + c(x/\delta)}4, \qquad b(x) =  1 + 0.5\cdot \cos(2\pi\cdot(x_1+2x_2) 
\]
and $c(x) = 1 -a^2(x)$.

\item In Figure \ref{figure exponent}, we consider the periodic potential
\[
W_\delta(x,u) = \big|u^2-1\big|^{p(x/\delta)}, \qquad p(x) = 1.5 + 8.5\cdot \cos^2\big(2\pi(x_2-\sin(2\pi x_1))\big).
\]
\end{enumerate}

The fourth experiment is the only one in which we do not know the homogenized potential. Due to the lack of a separable structure, the homogenized potential is harder to compute in this setting --  in all other cases, it is $(u^2-1)^2/4$.  For $W_{hom}(u) = (u^2-1)^2/4$, we know the optimal transition shape $\phi(x) = \tanh(x/\sqrt 2)$. We observe that $\tanh(1.5/\sqrt 2) \approx 0.79$ and $\tanh(3/\sqrt 2) \approx 0.97$. It is therefore reasonable to argue that the transition length scale to which $\delta$ should be compared lies somewhere between $3\eps$ and $6\eps$, i.e.\ between $0.075$ and $0.15$. 

The domain is $\Omega = (-2,2)^2$ with periodic boundary conditions (or equivalently, $\Omega= \R^2/ (4\Z+2)$). The first potential fails to be perfectly periodic at the boundary, but we do not observe significant effects here, and the theoretical results still apply to this situation. All other potentials are naturally defined on the torus.

\begin{figure}
%

\includegraphics[clip = true, trim = 1cm .5cm 1cm 1cm, height=3.2cm]{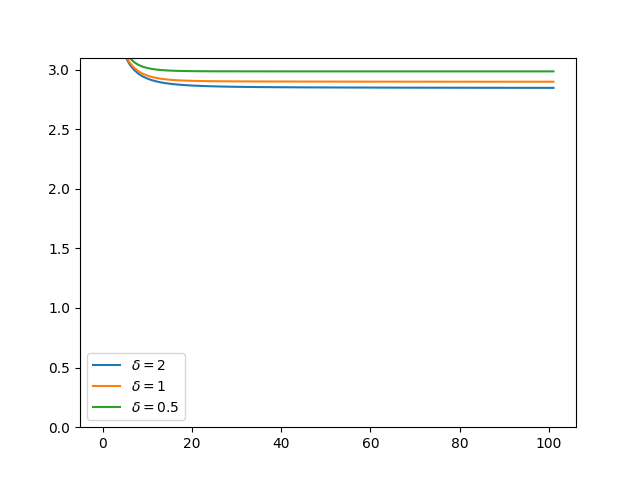}\hfill
\includegraphics[clip = true, trim = 2.4cm 1cm 3.7cm 1cm, height=3.2cm]{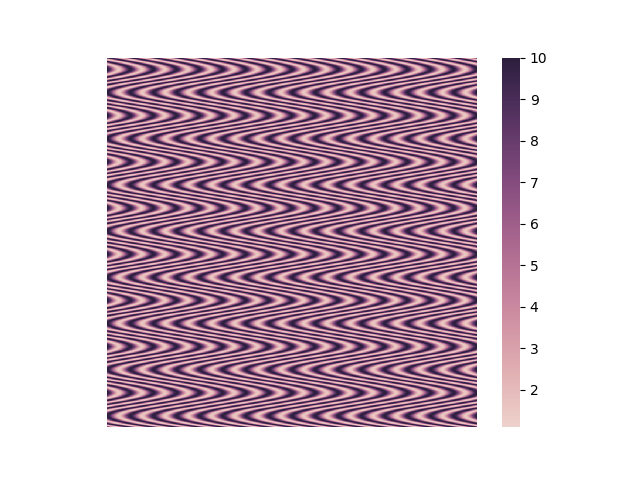}
\includegraphics[clip =  = true, trim = 2.4cm 1cm 3.7cm 1cm, height=3.2cm]{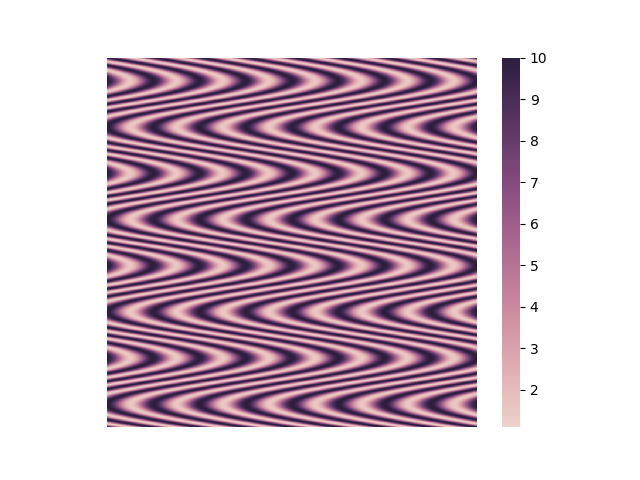}
\includegraphics[clip =  = true, trim = 2.4cm 1cm 1.5cm 1cm, height=3.2cm]{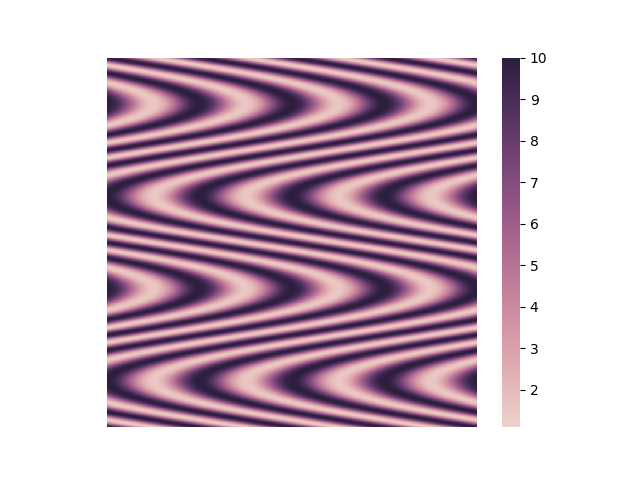}

\includegraphics[clip=true, trim = .8cm 0 0cm 0, height=3.2cm]{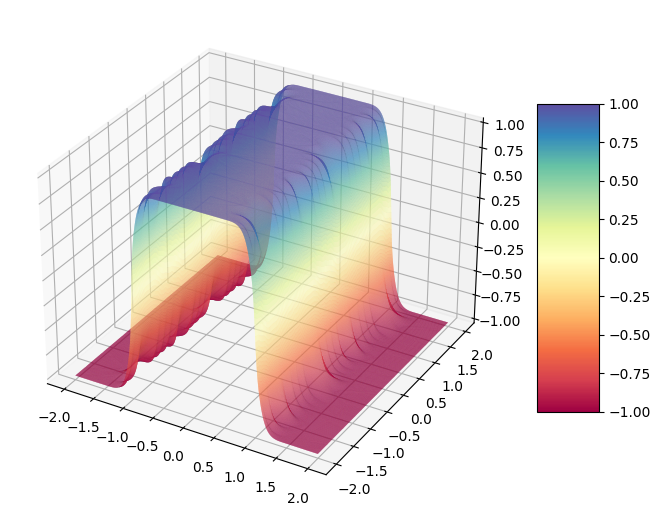}
\hfill\includegraphics[clip = true, trim = 2.4cm 1cm 3.7cm 1cm, height=3.2cm]{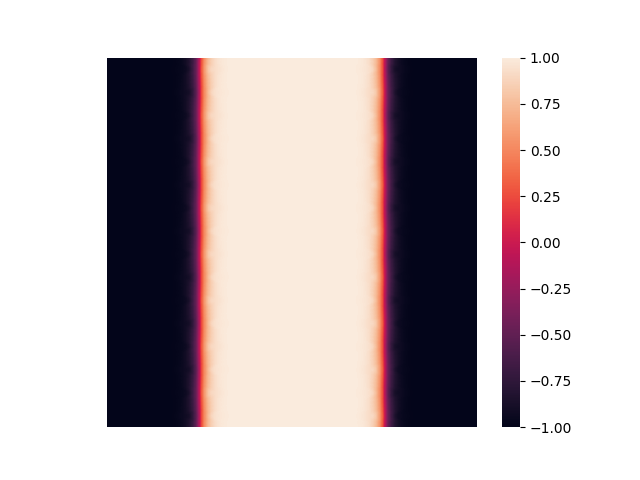}
\includegraphics[clip = true, trim = 2.4cm 1cm 3.7cm 1cm, height=3.2cm]{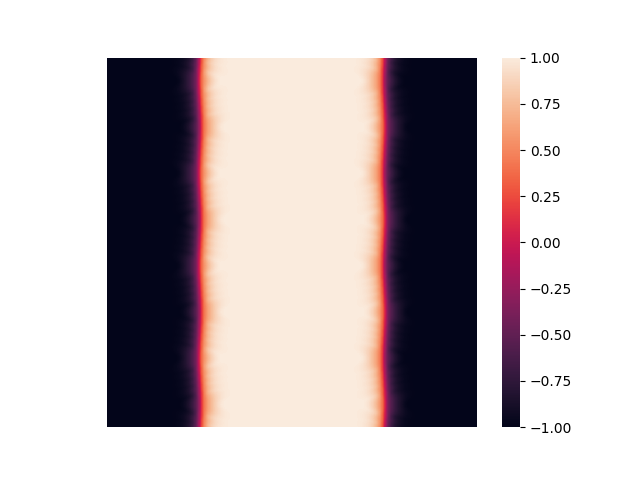}
\includegraphics[clip = true, trim = 2.4cm 1cm 1.5cm 1cm, height=3.2cm]{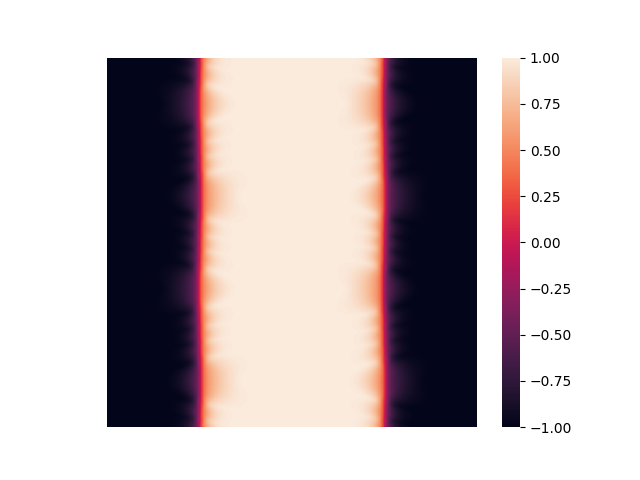}

\caption{\label{figure exponent}
Top row: The decay of the spatially inhomogeneous `Modica-Mortola' energy (left) and the spatially varying exponent $p(x/\delta)$ of the double-well potential for different values of $\delta$ in a `snake' setting. Bottom row: The terminal state in the same simulation for the homogenized potential and the terminal state of our simulations for various values of $\delta$ (corresponds to to top row). The right image is a different visualization of the rightmost contour plot.}
\end{figure}

{\bf Implementation.} In the simulation, we take 100 time steps of size $\tau = 10^{-3}$, i.e.\ following the gradient flow
\[
\partial_t u = \eps\,\Delta u - \frac{\partial_uW_\delta(x,u)}\eps
\] 
until time $t=0.1$. As the time-stepping scheme, we choose a Fourier space discretization
\[
\big(1+4\pi^2\tau\eps\,|\xi|^2\big)\,\F u_{t+1}(\xi) = \F\big(u_t - \tau\,\partial_uW_\delta(u_t, x) / \eps\big)(\xi)
\]
which treats the Laplacian implicitly and the double-well potential explicitly. The Fourier transform is computed as a fast Fourier transform on a $n\times n$ grid with $n^2=1,000,000$ nodes. The initial condition is the characteristic function $u_0 = 1_{\{(x,y) : -1< x<1\}}$ on the  periodic square $\Omega = (-2,2)^2$.

\begin{figure}
\begin{flushleft}
\includegraphics[clip=true, trim = .8cm 0 3.2cm 0, height=4.6cm]{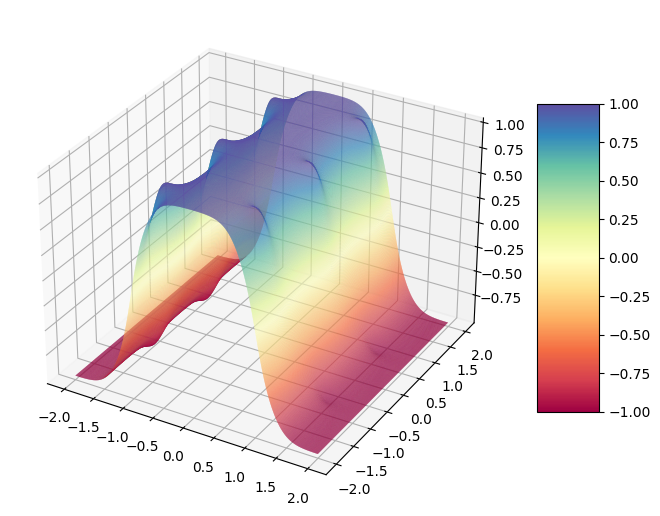}
\includegraphics[clip=true, trim = .8cm 0 3.2cm 0, height=4.6cm]{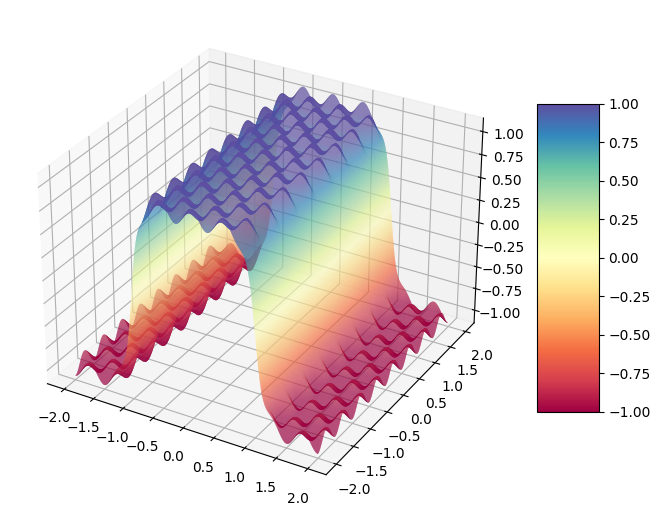}
\includegraphics[clip=true, trim = .8cm 0 .2cm 0, height=4.6cm]{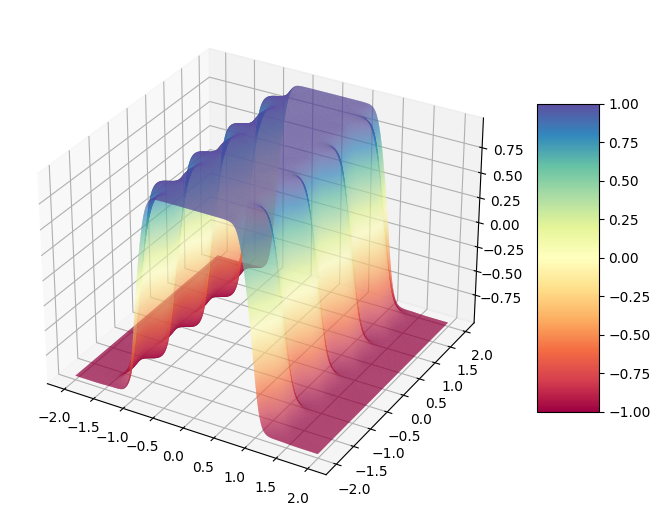}

\vspace{1mm}

\includegraphics[clip=true, trim = 1cm 0 1cm 1cm, width =4.48cm]{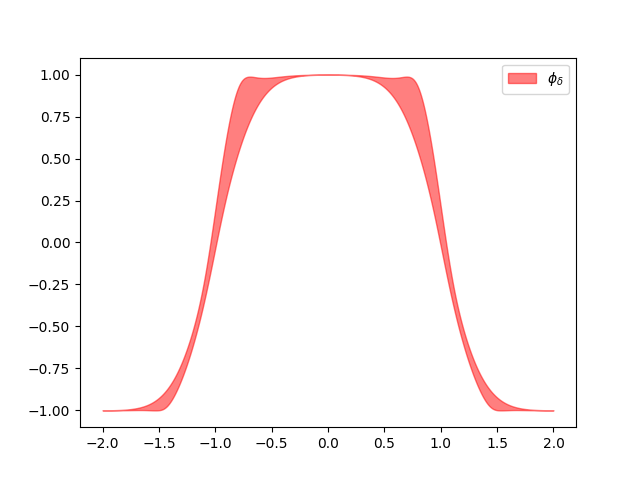}
\includegraphics[clip=true, trim = 1cm 0 1cm 1cm, width =4.48cm]{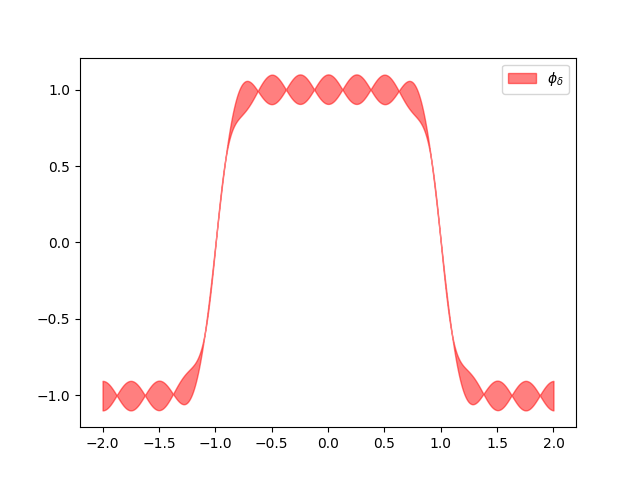}
\includegraphics[clip=true, trim = 1cm 0 1cm 1cm, width =4.48cm]{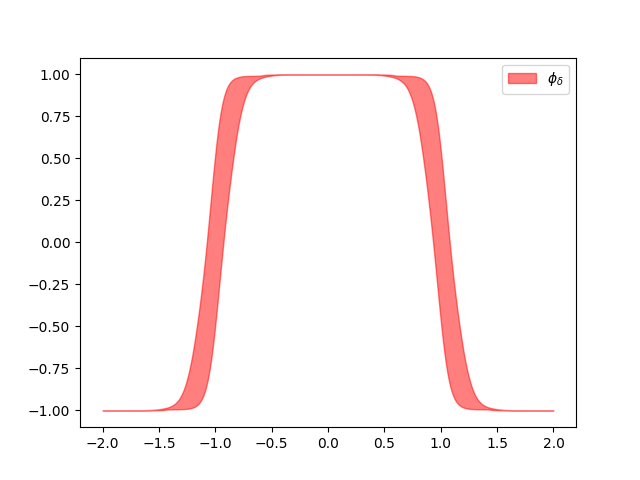}

\vspace{1mm}

\includegraphics[clip=true, trim = 2cm 5mm 2cm 1cm, width =4.48cm]{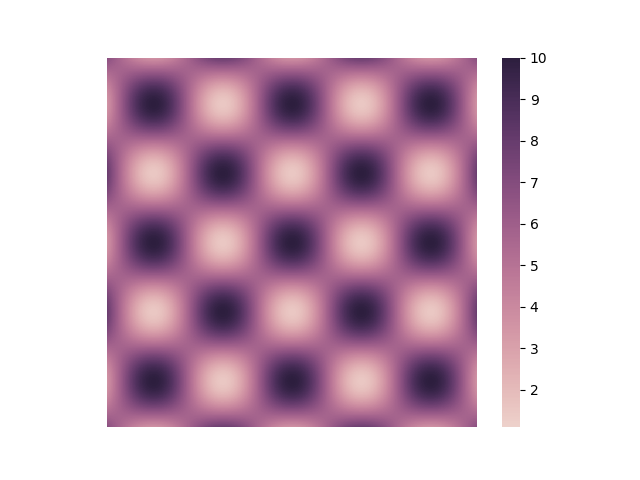}
\includegraphics[clip=true, trim = 2cm 5mm 2cm 1cm, width =4.48cm]{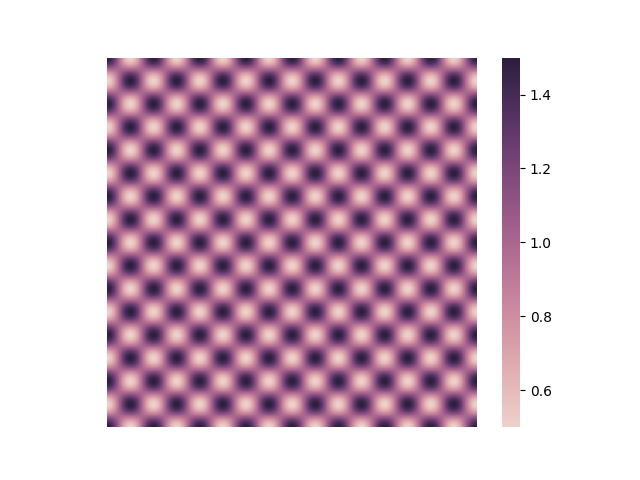}
\includegraphics[clip=true, trim = 2cm 5mm 2cm 1cm, width =4.48cm]{./hexweight_potential_delta_0.5.png}

\end{flushleft}
\caption{\label{pretty figure}
In this figure, we demonstrate that there are several different aspects of spatial inhomogeneity which disappear as we send $\delta\to 0$ for fixed $\eps$. The first row is always the terminal state $u$ of our simulation and the second row is always the space between the highest and lowest value that $u$ takes in $y$-direction for fixed $x$. In the left column, we consider a potential with spatially varying exponent $p(x)$ (plotted in the third row). In the second row, we observe a potential with spatially varying wells $a(x) \in [1/2, 3/2]$ (plotted in the third row). The potential in the third column is a product $W_{hom}(u) \cdot a(x)$ with a factor $a(x)\in [0,3]$ (plotted in the third row).\\
We find that (with a geometry parallel to the interface) the varying exponent barely affects the location of the interface,	 but has a large impact on its shape and steepness: The three `ridges' where $u$ makes a rapid transition coincide with the three places along the transition where the exponent $p(x)$ is close to $1$. On the opposite extreme end, the hexagonal potential shifts the interface to the region where $a(x)\approx 0$ to decrease length, but retains essentially the same transition shape and width everywhere. For the potential with varying wells, the interface is essentially unaffected and the effects are mostly noticeable in pure phase. This agrees with the results of \cite{cristoferi2023homogenization_delta_large}, who characterize the first order contribution to the $\Gamma$-expansion as a bulk integral (and the zeroth order contribution as an interface energy).
}
\end{figure}

{\bf Results and observations.} In all simulations, the gradient flow rapidly relaxes the jump in the initial condition to a finite energy transition on a length scale $\sim\eps$. For larger $\delta$, the energy decreases to lower values by exploiting the spatial microstructure of the potential. If $\delta$ is sufficiently small with respect to the transition length, then the interfaces are essentially straight and look like the transition of the homogenized potential between $-1$ and $1$, even if the wells of $W_\delta$ are not at $\pm 1$.
As $\delta\to 0$ and for small $\eps$, the energy approaches $2\cdot \frac{2\sqrt 2}3 = \int_{-1}^1 \sqrt{2\,W_{hom}(u)} \d u \approx 1.89$ in all experiments where $W_{hom}(u) = (u^2-1)^2/4$. This is the correct length for two opposing interfaces on the torus as we are using {\em average integrals} in the energy \eqref{eq energy numerical} rather than unnormalized integrals.

For potentials $W$ which only depend on the phase parameter $u$, it is well-understood that the gradient flow `relaxes' $u$ from an initial condition $u_0:\Omega\to(-1,1)$ to a function which take values mostly in the potential wells $\pm 1$ and transitions on a length scale $\eps$ on a fast time scale $\sim \eps|\log\eps|$ where the `ODE' $\dot u = -W'(u)$ dominates the dynamics. On the `slow' time-scale $\sim1/\eps$, the interface between the phases $\{u\approx 1\}$ and $\{u\approx-1\}$ moves by mean curvature flow (MCF) \cite{ilmanen1993convergence, mugnai2011convergence, fischer2020convergence}. In one dimension -- or for interfaces along parallel straight lines in two dimensions -- the dynamics become exponentially slow in $1/\eps$ \cite{carr1989metastable, fusco1989slow, bronsard_kohn} and indeed, also for potentials with spatial inhomogeneity, two parallel lines appear to correspond to a meta-stable state, especially if $\delta$ is small.

There are several different mechanisms of inhomogeneity which all disappear as $\delta/\eps\to 0^+$, including the shape and width of the transition between the potential wells, the geometry of the level set $\{u=0\}$, and the behavior in pure phase. We visualize this observation in Figure \ref{pretty figure}. 
Notably, the homogenization of potentials with varying wells requires $\delta$ to be much smaller than $\eps$, as predicted by the theoretical analysis.

\newpage
\bibliographystyle{alpha}
\bibliography{./mm.bib}

\end{document}